\documentclass[11pt,letterpaper]{article}
\usepackage[english]{babel}
\usepackage{amsmath}
\usepackage{amsfonts, mathrsfs}
\usepackage{amssymb}
\usepackage{verbatim}
\usepackage{graphicx}
\usepackage{color}
\usepackage{url}

\numberwithin{equation}{section}
\newtheorem{theo}{Theorem}[section]

\newtheorem{prop}[theo]{Proposition}
\newtheorem{lemma}[theo]{Lemma}

\newtheorem{assum}[theo]{Assumption}
\newtheorem{defn}[theo]{Definition}

\newtheorem{remark}[theo]{Remark}
\newenvironment{proof}[1][Proof]{\textbf{#1.} }{\ \rule{0.5em}{0.5em}}

\begin{document}

\title{Parallel queues with synchronization}

\author{Mariana Olvera-Cravioto \\  {\small Columbia University}
\and
Octavio Ruiz-Lacedelli \\  {\small Columbia University}
}

\maketitle

\begin{abstract}
Motivated by the growing interest in today's massive parallel computing capabilities we analyze a queueing network with many servers in parallel to which jobs arrive a according to a Poisson process. Each job, upon arrival, is split into several pieces which are randomly routed to specific servers in the network, without centralized information about the status of the servers' individual queues. The main feature of this system is that the different pieces of a job must initiate their service in a synchronized fashion. Moreover, the system operates in a FCFS basis. The synchronization and service discipline create blocking and idleness among the servers, which is compensated by the fast service time attained through the parallelization of the work.  We analyze the stationary waiting time distribution of jobs under a many servers limit and provide exact tail asymptotics; these asymptotics generalize the celebrated Cram\'er-Lundberg approximation for the single-server queue.  

\vspace{5mm}

\noindent {\em Kewywords:} Queueing networks with synchronization, many servers queue, Cram\'er-Lundberg approximation, high-order Lindley equation, cloud computing, MapReduce, weighted branching processes, stochastic fixed-point equations, large deviations. 


\end{abstract}

\section{Introduction}

Cloud computing is an emerging paradigm for accessing shared computing and storage resources over the internet.  ``Clouds" consist of hundreds of thousands of servers that provide scalable, on-demand data storage and processing capacity to end users.  Operators of these large facilities achieve economies of scale and can make efficient use of their computing infrastructure by optimizing how tasks are processed across an interconnected and distributed computer network.  End users of cloud computing benefit from reduced capital expenditures and from the flexibility that the scalable paradigm offers.  Worldwide demand for cloud computing has experienced rapid growth that is expected to continue into the foreseeable future.

Motivated by this rapidly emerging phenomenon, we analyze a queueing model for a large network of parallel servers.  Throughout the paper we use the generic term server to  represent a computing unit, e.g., a computer or a processor in the network. Jobs arrive to the network at random times and are split at the time of arrival into a number of pieces. These pieces are then immediately assigned to randomly selected servers, where they join the corresponding queues. The main reason for assigning these job fragments to specific servers is that maintaining a single centralized queue is not scalable in systems of this size, and keeping information about the individual status of each queue to make informed routing decisions can be costly. In systems where data locality is important, or where certain tasks need to be done at specific servers, the random routing can be used to approximately model storage, or specific resources, that are randomly spread out through the network.  The service requirements of each of the pieces of a job are allowed to be random and possibly dependent. The main distinctive features of this model are: 1) all the pieces of a job must begin their service at the same time, i.e., in a synchronized fashion; 2) jobs are processed in a first-come-first-serve (FCFS) basis, i.e., each of the individual queues at the servers follow a FCFS service discipline. We refer to these two characteristics as the {\em synchronization} and {\em fairness} requirements. Figure \ref{F.Cloud} depicts our model. 

The fairness requirement is common in many queueing systems where jobs originate from different users, while the job synchronization is a distinctive characteristic of this model that allows us to incorporate the need to exchange information among the different pieces of a job during their processing. This is certainly the case for many scientific applications that involve simulations of complex systems, including: wireless networks, neuronal networks, bio-molecular and 
biological systems. Alternatively, our model also provides an approximation for systems where job fragments need to be joined after being processed, which will be discussed in detail in Section \ref{S.MapReduce}. We point out that the synchronization of the different pieces of a job can in principle be attained without the need of centralized information, since the specific server assignment occurs upon arrival, and each piece would only need to keep track of its ``sibling" fragments, e.g., once a fragment is ready to initiate its service it can notify its siblings, and the last one to do so determines when the job can start processing.

The fairness and synchronization requirements create, nonetheless, blocking and idleness that are not present in other distributed systems, e.g., multi server queues where the different fragments of a job can be processed independently, and can therefore be thought of as batch arrivals. This lack of efficiency is compensated by the service  speed attained through the parallel processing, which can be considerable for very large jobs. To illustrate this gain in processing speed we have compared in Section \ref{S.Numerics} (Table \ref{T.CompareThree})  the sojourn times of jobs in our model and in a comparable multiserver queue where jobs are not split into pieces. A detailed discussion about this comparison can be found in Section~\ref{S.Numerics}, but for now it suffices to mention that our model dramatically outperforms the non-distributed system in spite of the blocking and suboptimal routing.

\begin{figure}[h] 
\begin{center}
\includegraphics[scale=0.7, bb = 50 480 550 750, clip]{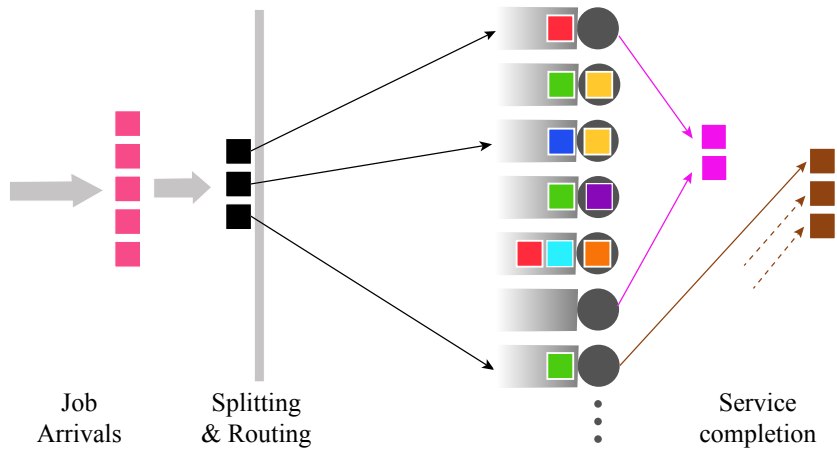}
\caption{Queueing model for a server cloud where jobs need to be synchronized. In this figure, the yellow, purple and orange jobs are being processed, while the pink and brown jobs have completed their service; two of the three brown pieces were processed at servers not depicted in the diagram. Note that the last server at the bottom will remain idle until the yellow and purple jobs complete their service, and the first server at the top will need to wait for both the orange and light blue jobs to be done before starting to process the red job.  The blue job in queue at the third server from the top has only one piece and can begin its processing as soon as the yellow job leaves.} \label{F.Cloud}
\end{center}
\end{figure}

We analyze in this paper the stationary waiting time of jobs (excluding service) in an asymptotic regime where the arrival rate of jobs and the number of servers grow to infinity, but the sizes of jobs and the service requirements of the individual fragments remain constant. For simplicity, we will refer to this type of limit as a ``many server asymptotic regime", not to be confused with the Halfin-Whitt regime used in multiserver queueing systems (\cite{Whitt_2002}).  In particular, after establishing sufficient conditions for the stability of the finite system, we show that the limiting stationary waiting time $W$ is given by the endogenous solution to the following high-order Lindley equation: 
\begin{equation} \label{eq:AdditiveMax}
W \stackrel{\mathcal{D}}{=} \left(  \max_{1 \leq i \leq N} \left( W_i + \chi_i - \tau_i \right) \right)^+,
\end{equation}
where the $\{W_i\}_{i \in \mathbb{N}}$  are i.i.d. copies of $W$, independent of $(N, \chi_1,\tau_1, \chi_2, \tau_2, \dots)$, $N$ is the number of pieces of a job (allowed to be random), $\tau_i$ is the limiting inter arrival time between piece $i$ and the job immediately in front of it at its assigned queue, and $\chi_i$ is the service time of the fragment of the job in front of the $i$th piece; $\stackrel{\mathcal{D}}{=}$ stands for equality in distribution and $x^+ = \max\{0, x\}$. Note that for $N \equiv 1$, \eqref{eq:AdditiveMax} reduces to the classical Lindley equation, satisfied by the GI/GI/1 queue.
Recursion \eqref{eq:AdditiveMax} was termed ``high-order Lindley equation" and studied in the context of queues with synchronization in \cite{Kar_Kel_Suh_94}, although only for deterministic $N$.

Moreover, by applying the main result in \cite{Jel_Olv_14} for the maximum of the branching random walk,  we provide the exact asymptotics for the tail distribution of $W$. To explain the significance of this result it is worth considering first a single-server queue with renewal arrivals and i.i.d.~service requirements, for which it is well known that the stationary waiting time distribution, $W^\text{GI/GI/1}$, satisfies
$$P\left( W^{GI/GI/1} > x \right) = P\left( \max_{k \geq 1} \, S_k > x \right),$$
where $S_k = X_1 + \dots + X_k$ is a random walk with i.i.d.~increments satisfying $E[X_1] < 0$. Using the ladder heights of $\{S_k\}$ and renewal theory yields the celebrated Cram\'er-Lundberg approximation
$$P\left( W^\text{GI/GI/1} > x \right) \sim H_\text{GI/GI/1} \, e^{-\theta x}, \qquad x \to \infty,$$
where $0 < H_\text{GI/GI/1} < \infty$ is a constant that can be written in terms of the limiting excess of the renewal process defined by the ladder heights, and $\theta > 0$, known as the Cram\'er-Lundberg root, solves $E[e^{\theta X_1}] = 1$ and satisfies $0 < E[X_1 e^{\theta X_1}] < \infty$ (see \cite{Asm2003}, Chapter XIII, for more details). Throughout the paper, $f(x) \sim g(x)$ as $x \to \infty$ stands for $\lim_{x \to \infty} f(x)/g(x) = 1$.

Back to the asymptotic behavior of $W$ in this paper, Theorem 3.4 in \cite{Jel_Olv_14} states that for $\theta > 0$ satisfying the conditions  
\[
E\left[ \sum_{i=1}^N e^{\theta (\chi_i- \tau_i)} \right]=1 \qquad \text{and} \qquad  0<E\left[\sum_{i=1}^N (\chi_i - \tau_i) e^{\theta (\chi_i - \tau_i)} \right]<\infty,
\]
we have that  
\begin{equation}
\label{eq:genCL}
P(W>x) \sim H e^{-\theta x}
\end{equation}
as $x\rightarrow \infty$ for some constant $0<H<\infty$. In other words, the queueing system with parallel servers and synchronization requirements in this paper naturally generalizes the single-server queue, as well as its Cram\'er-Lundberg approximation.

\subsection{Split-merge queues and MapReduce} \label{S.MapReduce}

As mentioned earlier, our model depicts a queueing network with many parallel servers where incoming jobs consist of a random number of pieces, to be processed in parallel by a randomly chosen subset of servers, under the constraint that all the pieces must begin their service simultaneously. This setup is very closely related to a queueing network known as a {\em split-merge} queue  (\cite{Harr_Zer_03, Leb_Knott_07, Leb_Din_Knott_08, Tsi_Knott_Harr_14}). The main difference between a split-merge queue and the model described in this paper is that in the former the synchronization occurs once all the pieces of a job have completed their service. More precisely, job fragments are allowed to start their processing as soon as their assigned server becomes available, but will continue blocking it, even after having completed their service, until all other pieces of the same job have completed theirs. Split-and-merge queues have been used for the analysis of ``redundant arrays of independent disks (RAID)" in \cite{Harr_Zer_03, Leb_Din_Knott_08}, and constitute a special case of a fork-join network where there are no output buffers. Hence, split-merge queues provide natural upper bounds for their corresponding fork-join counterparts. One important feature of our model compared to split-merge queues, or even the more general fork-join networks, is that much of the existing literature assumes that all jobs have the same number of pieces, which is equal to the number of servers in the network (usually small). Our model, which is meant to model networks with a very large number of parallel servers, allows jobs to have different number of pieces with heterogeneous service requirements. 

A popular distributed algorithm that can be modeled using a split-merge queue is MapReduce (\cite{Dean_Ghem_04}), and its open source implementation, Hadoop. The main idea of MapReduce/Hadoop is to divide large data sets into smaller units, then process these smaller units on a large number of parallel servers and finally assemble  the partial answers into the final solution.  The initial phase of this framework, called the mapping, divides a new job into tasks/files of similar size, e.g., 64 (or 128) megabytes (MB) in size. Irrespective of the size of the original job, all smaller tasks are of equal size - 64 (or 128) MBs, but the number of these tasks and the servers to which they are assigned depends on the original job size. After (or in some implementations during) the execution of the mapping phase the system begins a shuffling phase, which is then followed by a reducing phase, again on a number of parallel servers. The reduce phase merges the partial answers from the processed tasks into  one final answer.  Often, the completion of a job consists of several repetitions in sequence of the map-reduce process. 

To explain the similarities between a split-merge queue and MapReduce it is worth elaborating some more on its synchronization and blocking characteristics.  In the original FCFS implementation, if there are two jobs A and B arriving in that order, all mapping tasks for job A will execute before any mapping tasks from job B can begin. As tasks from job A are finishing their mapping phase and moving on to the reduce phase, job B mapping tasks are scheduled.  This is similar to the synchronized server assignment in a split-merge queue. 
Furthermore, due to the observation that reducers need to have all the outputs from the mapping phase in order to perform their work, job B reducers cannot start until all job A reducers and all job B mappers have finished.  
Therefore, there is a natural queueing system with blocking of servers.  In the Hadoop framework, a job needs both mappers and reducers available in order to start processing, and the reducers can only begin once all (some) mappers have finished, hence blocking servers for other jobs that arrived at a later time. The blocking, or starvation, problem is a known shortcoming of MapReduce (see, e.g., \cite{Tan_Meng_Zhang_12}). 

It follows that a MapReduce implementation with FCFS scheduling can be well described by a split-merge queue with synchronization whenever the servers remain blocked until all the subtasks are rejoined. Section~\ref{S.Numerics} shows numerical evidence supporting the use of our model for approximating a split-and-merge queue. The main advantage of using our model for this purpose is its mathematical tractability, since as our main results show, the waiting time distribution as well as all its moments can be accurately estimated when the number of servers is large. Our model can also incorporate more realistic features such as jobs of different sizes and dependent, non-identically distributed service requirements for their subtasks.

\subsection{Related Literature} \label{S.Literature}

The model studied in this paper has connections to a vast literature, and therefore we restrict this section to only a few closely related models and some recent work on the applications mentioned above.

\begin{enumerate}
\item {\bf Distributed queueing models with synchronized service.} We start with the model considered in \cite{Green_80}, which studies a queueing system where each job requires a synchronous execution on a random number of parallel servers. Some applications mentioned there are: the deployment of fire engines in firefighting, jury selection, and the staffing of surgeons and medical personnel in emergency surgery. The main difference in \cite{Green_80} from our setup, besides the restriction to i.i.d. exponentially distributed service times, is that we assign the pieces of a job to specific servers at the time of arrival, while the model in \cite{Green_80} waits until the required number of servers is free and then assigns the pieces to these servers.  We point out that this is equivalent to having perfect information about the workloads at each server and routing the pieces to the servers with the smallest workloads. Therefore, the model in \cite{Green_80} provides a benchmark that can be used to quantify the value of having centralized information.  

Some of the ideas used in the proof of the main result in this paper are borrowed from \cite{Kar_Kel_Suh_94}, where the authors considered a queueing system with $m$ different types of servers and $n$ identical servers of each type (for a total of $m\times n$ servers), and where each arriving job requires service from exactly one server of each type, i.e., each job needs $m$ parallel servers, and is assigned upon arrival to one of the $n$ possible choices for each type. Theorem~2 in \cite{Kar_Kel_Suh_94} shows that the steady-state distribution of the waiting time converges weakly, as $n \to \infty$, to the endogenous solution of the high-order Lindley equation \eqref{eq:AdditiveMax} with $N \equiv m$. Besides allowing $N$ to be random and the service requirements of the fragments of a job to be dependent, this paper shows not only the weak convergence of the steady-state waiting time, but also of all its moments. The proof technique used in this paper, which is based on a coupling using the Wasserstein distance, is new, and is responsible for the stronger mode of convergence.

A third related model is the one considered in \cite{Bacc_Foss_11}, which can be thought of as a stylized version of our queueing network were server assignment is not done uniformly at random, but rather according to a distribution on specific subsets of servers, e.g., blocks of adjacent or closely located servers. The setting there corresponds to all the fragments of a job having identical service requirements, but provides interesting insights into the existence of stationary distributions for different server assignment rules.

\item {\bf Fork-join queues with synchronization.} The analysis of distributed computer systems with synchronization constraints such as MapReduce/Hadoop, RAID, or even online retail, constitutes an active area of research. From a queueing theory perspective, perhaps the most widely used model is that of a fork-and-join queue (\cite{Harr_Zer_03, Leb_Knott_07, Tsi_Knott_Harr_14, Thomasian_14}), which as pointed out earlier differs from our model in two ways: 1) the synchronization occurs once all the pieces of a job have completed their service, and 2) there may be output buffers after each queue to prevent the servers from being blocked. Moreover, fork-join queues are in general difficult to analyze, with much of the literature focusing on approximating mean sojourn times of jobs (\cite{Nel_Tan_88, Bacc_Mak_Shw_89, Var_Mak_94, Tho_Tan_94, Varki_99, Varki_01, Leb_Din_Knott_08}). Exact analytical results exist only for a 2-server system with i.i.d.~exponential service requirements (\cite{Fla_Hah_84, Nel_Tan_88}), and most of the remaining approximations do not scale well as the number of servers grows large. Heavy-traffic approximations can be found in \cite{Varma_90, Nguyen_93}. 

More tangentially related, we also mention that a considerable amount of work related to the modeling of MapReduce and other distributed computer platforms is done from the scheduling perspective. We refer the interested reader to \cite{Tan_Meng_Zhang_12, Zah_etal_09, Lin_etal_13} and the references therein. 
% Add Yuan Zhong later...

\item {\bf Resource allocation problems.} Although only briefly, we mention that the model in this paper is also related to virtual path allocation problems in communication networks, where incoming calls request a specific set of links in the network (virtual path) to establish a communication channel, and if any of the requested links is unavailable at that time the call is lost (see, e.g., \cite{Whitt_85, Kelly_87}). In a queueing interpretation of this model, one can think of calls as jobs, links as servers, and the duration of the call as the common service requirement of the job fragments. Existing work in this area has been focused on the blocking-loss model, where the main performance measure is the loss probability. Hence, our model can be thought of as a variation of this model where queueing is allowed and the number of links is very large. More generally, our model can provide valuable insights into a wider class of multi-resource allocation problems, such as those appearing in service engineering and consulting project management, provided the number of resources is sufficiently large to justify the limiting regime.

\end{enumerate}

The remainder of the paper is organized as follows. Section~\ref{S.Model} contains the mathematical description of our queueing model; Section~\ref{S.MainAnalysis} describes the analysis of the stationary waiting time of jobs in the network, with the main result of this paper in Section~\ref{SS.MainResult} and the tail asymptotics of the limiting waiting time in Section~\ref{SS.TheLimit}. Section~\ref{S.Numerics} contains the numerical experiments mentioned earlier, and Section~\ref{S.Conclusions} gives some concluding remarks.  Finally, the proofs of the two theorems are given in Section~\ref{S.Proofs}.

\bigskip
\section{Model description} \label{S.Model}

We consider a sequence of queueing networks  indexed by their number of servers, $n$. Each of the $n$ servers are identical and operate in parallel.  Arrivals to the $n$th network occur according to a Poisson process with rate $\lambda n$ for some parameter $\lambda > 0$. Each job, upon arrival to the network, is split into a random number of pieces, usually proportional to the total service requirement of the job. The size of a job, i.e., the number of pieces into which it is split, is determined by some distribution $f_n(k)$, $k = 1, 2, \dots, m_n$, where $m_n$ is a bound on the number of pieces a job can have and is chosen to satisfy $m_n \leq n$; this condition ensures that each piece can be routed to a different server. Once a job has been split, say into $k$ pieces, its fragments are routed randomly to $k$ different servers in the network (i.e., with all $n!/(n-k)!$ possible assignments equally likely), forming a queue at their assigned servers.  An equivalent way of describing the arrival of jobs into the $n$th network is to use the thinning property of the Poisson process and think of independent Poisson processes, each generating jobs of size $k$, $k = 1, 2, \dots, m_n$, at rate $f_n(k)\lambda n.$

The service times of the different fragments of a job are assumed to have a general distribution, although the stability condition for the model will implicitly impose that they have finite exponential moments. Moreover, they are not assumed to be identically distributed and are allowed to be dependent, although we do require that they be independent of the number of pieces. More precisely, a typical job has $\hat N$ pieces having service requirements $(\chi^{(1)}, \dots, \chi^{(\hat N)})$, where $\hat N$ has distribution $f_n$, and  $\chi^{(j)}$ is independent of $\hat N$ for each $j$. Since the random routing  eliminates information about the order of the pieces, the relevant distribution that will appear in the analysis of the model is that of a randomly chosen fragment, which we denote $B$. The randomness of the fragments' service requirements can be used to include rounding effects and small variations on the type of processing that they need. The sizes of jobs and of the service requirements of their fragments are assumed to be independent of the arrival process. 

In order to model the synchronization and fairness characteristics of the network, we will assign to each job a tag (not to be confused with the label that will be introduced later). More precisely, a job having $k$ pieces receives a tag of the form $(s_1, s_2, \dots, s_k)$, $s_i \in \{1, 2, \dots, n\}$ for all $i$, $s_i \neq s_j$ for $i \neq j$, representing the different servers to which its fragments are sent for processing.

\begin{defn} \label{D.Predecessor}
We say that a job having tag ${\bf r} = (r_1, r_2, \dots, r_l)$ is a {\em predecessor} of a job having tag ${\bf s} = (s_1, s_2, \dots, s_k)$ if it arrived before the job having tag ${\bf s}$ and they have at least one server in common (i.e., $r_i =  s_j$ for some $1 \leq i \leq l$ and $1 \leq j \leq k$). We use the term {\em immediate predecessor} if there are no pieces of other jobs in between the two jobs at the server they have in common. 
\end{defn}

In terms of this definition, the {\em synchronization} rule is that the job having tag ${\bf s} = (s_1, s_2, \dots, s_k)$ cannot begin its service, which is to be done in parallel by servers $s_1, s_2, \dots, s_k$, until all its immediate predecessors have completed their service. The {\em fairness} rule says that if the job with tag ${\bf r}$ is a predecessor of the job with tag ${\bf s}$, then it will begin its service before the job with tag ${\bf s}$ does. 

Formally, we can think of the $n$th system as a superposition of $\sum_{k=1}^{m_n} \frac{n!}{(n-k)!}$ independent marked point processes. Each of these processes generates jobs of size $k$ with server assignments $(s_1, \dots, s_k)$, according to a Poisson process with rate $f_n(k) \lambda n/(n!/(n-k)!)$, and having i.i.d.~marks \linebreak $\left\{ (\chi_i^{(1)}, \dots, \chi_i^{(k)}) \right\}_{i \geq 1}$ corresponding to the service requirements of the fragments. To establish stability, we follow the standard queueing theory technique of assuming that at time $t_0 < 0$ there are no jobs in the network, and then look at the waiting time of the first job to arrive after time zero. We will refer to this job as the ``tagged" job, and we will prove that its waiting time, after having taken the limit $t_0 \to -\infty$, is finite almost almost surely. The stability will then follow from Loynes' lemma and Palm theory (see \cite{Bacc_Brem_2003} for more details on this general technique). 

To analyze the waiting time of the tagged job we look at a graph containing all the information of which jobs need to complete their service before the tagged job can initiate its own. We now describe how to construct such graph, called in \cite{Kar_Kel_Suh_94} a predecessor graph.

\subsection{The predecessor graph}

To construct the predecessor graph we look at time in reverse, starting from the time the tagged job arrived, say $T_1 \geq 0$, and ending at time $t_0$. The tagged job, which we will label $\emptyset$, is split into a random number of pieces, say $\hat N_\emptyset = \hat N_\emptyset (n)$, where $\hat N_\emptyset$ is distributed according to $f_n$. Each of these pieces will be routed to one of the $n$ servers in the network, where it will either find the server empty or join a queue. Suppose that the tagged job needs to be processed by servers $(s_1, \dots, s_{\hat N_\emptyset})$, and recall from Definition~\ref{D.Predecessor} that any job that is directly in front of the queue at any of the servers $s_i$, $1 \leq i \leq \hat N_\emptyset$ is an immediate predecessor of the tagged job. To construct the first set of edges in the graph we draw an edge from the tagged job to its immediate predecessors. Moreover, each edge is assigned a vector of the form $(\hat \tau_i, \chi_i)$, $1 \leq i \leq \hat N_{\emptyset}$, where $\hat \tau_{i}$ is the inter arrival time between the tagged job and its $i$th immediate predecessor, and $\chi_{i}$ is the service requirement of the piece of the immediate predecessor that is in front of the corresponding piece of the tagged job. Also, if a job is an immediate predecessor of more than one fragment of the tagged job, say it requires service at servers $s_i$ and $s_j$, then $\hat \tau_{i} = \hat \tau_{j}$, although we may still have $\chi_{i} \neq \chi_{j}$ with $\chi_i, \chi_j$ possibly dependent. Finally, if a piece of the tagged job finds its server empty upon arrival, then there is simply no edge to be drawn. Hence, the number of outbound edges of the tagged job is smaller or equal than $\hat N_\emptyset$.

Iteratively, once we have identified all the immediate predecessors of the tagged job we repeat the process described above with each one of them. We will call the predecessor graph $\mathcal{G}_n(t_0)$, since it will depend on both the number of servers $n$ and the time $t_0$ at which the system starts empty. Since $\mathcal{G}_n(t_0)$ will resemble a tree, it will be useful to use tree notation to refer to the predecessors of the tagged job. More precisely, let $\mathbb{N}_+ = \{1, 2, 3, \dots\}$ be the set of positive integers and let $U = \bigcup_{r=0}^\infty (\mathbb{N}_+)^r$ be the set of all finite sequences ${\bf i} = (i_1, i_2, \dots, i_r) \in U$, where by convention $\mathbb{N}_+^0 = \{ \emptyset\}$ contains the null sequence $\emptyset$. 
To ease the exposition, for a sequence ${\bf i} = (i_1, i_2, \dots, i_k) \in U$ we write ${\bf i}|t = (i_1, i_2, \dots, i_t)$, provided $k \geq t$, and  ${\bf i}|0 = \emptyset$ to denote the index truncation at level $t$, $k \geq 0$. 
To simplify the notation, for ${\bf i} \in \mathbb{N}_+$ we simply use ${\bf i} = i_1$, that is, without the parenthesis. Also, for ${\bf i} = (i_1, \dots, i_k)$ we will use $({\bf i}, j) = (i_1,\dots, i_k, j)$ to denote the index concatenation operation, if ${\bf i} = \emptyset$, then $({\bf i}, j) = j$. 

\begin{figure}[h]
\begin{center}
\includegraphics[scale=0.7, bb = 50 480 560 750, clip]{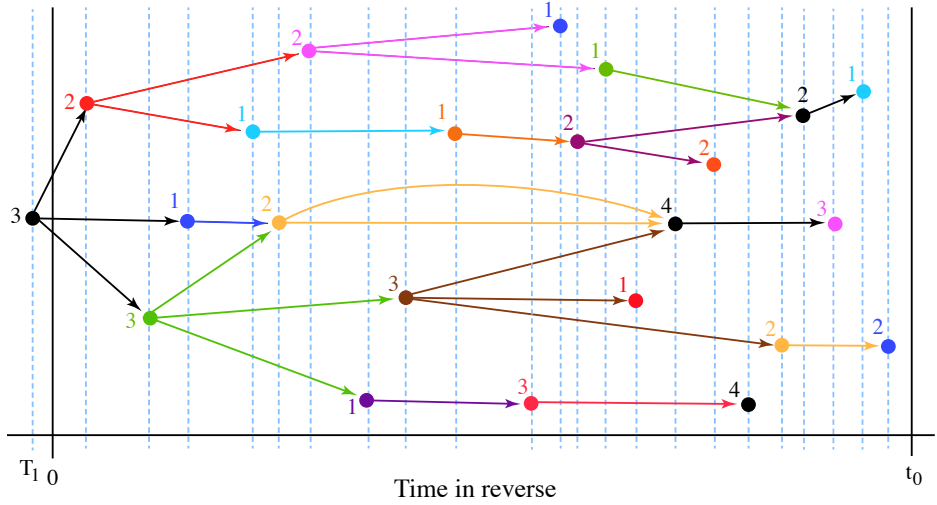}
\caption{The predecessor graph $\mathcal{G}_n(t_0)$. The numbers in each node indicate the size of the job (number of pieces). Some nodes have fewer outbound edges than the size of the job, meaning that the corresponding piece found its server empty upon arrival. Nodes with multiple inbound edges correspond to jobs that are immediate predecessors to more than one job in the graph. Vertical lines indicate the time of arrival of each job; service requirements for the pieces of a job can be thought of as ``edge attributes" and cannot be read from the graph. This graph is consistent with Figure~\ref{F.Cloud} by letting the tagged job be the three piece black one.}  \label{F.PredecessorGraph}
\end{center}
\end{figure}

Now recall that $\emptyset$ denotes the tagged job and label its immediate predecessors $i$, with $1 \leq i \leq \hat N_\emptyset$. The jobs in the next level of predecessors will have labels of the form $(i_1, i_2)$, and in general, any job in the predecessor graph will have a label of the form ${\bf i} = (i_1, i_2, \dots, i_k)$, $k \geq 1$. With this notation, $\hat N_{\bf i}$ denotes the number of pieces that the job with label ${\bf i}$ in the graph is split into, $\hat \tau_{({\bf i},j)}$ will denote the inter arrival time between job ${\bf i}$ and its $j$th immediate predecessor (a job with label $({\bf i},j)$), and $\chi_{({\bf i},j)}$ denotes the service requirement of the fragment immediately in front of the queue of the $j$th piece of job ${\bf i}$.  Note that the tag of a job, which contains the specific server assignments, allows us to identify the immediate predecessors of a given job, but it plays no role afterwards. Therefore, we will use the labels, not tags, to identify jobs in the predecessor graph. See Figure \ref{F.PredecessorGraph}.

We point out that in case a job is an immediate predecessor to more than one job in the graph (or to more than one piece of the same job), the corresponding edges will merge into the common predecessor. Moreover, in this case, the common predecessor is assigned more than one label (e.g., if a job is an immediate predecessor to both jobs ${\bf i} = (i_1, \dots, i_k)$ and ${\bf j} = (j_1, \dots, j_l)$, then such job can be identified by two different labels, one of the form $({\bf i}, s)$ and another of the form $({\bf j}, t)$). Furthermore, the merged paths and the subgraph they define from that point onwards will have multiple labels as well. See Figure~\ref{F.Collisions}.

\begin{figure}[h]
\begin{center}
\includegraphics[scale=0.7, bb = 60 100 540 240, clip]{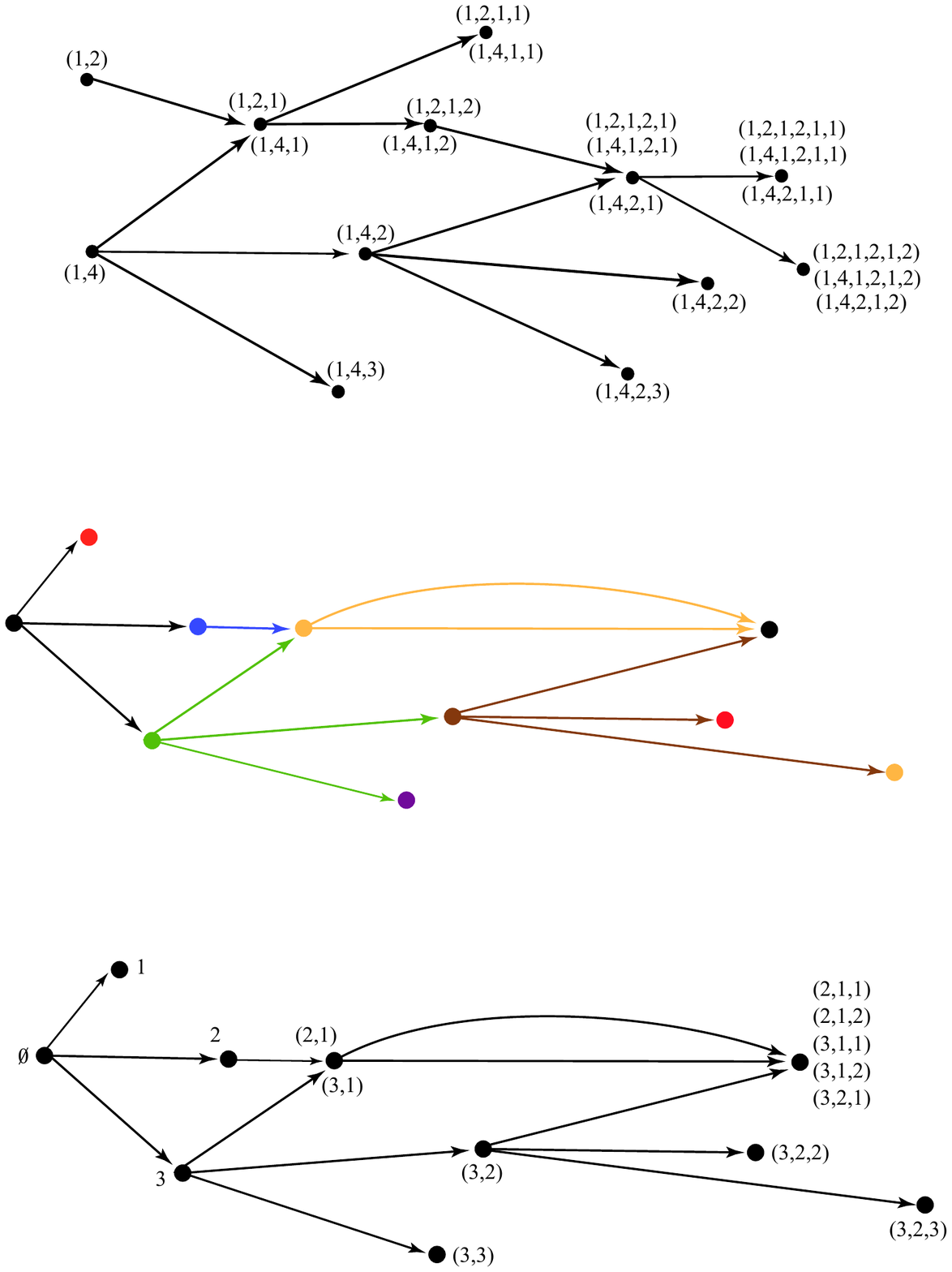}
\end{center}
\caption{Multiple labels due to common immediate predecessors. Excerpt from the predecessor graph in Figure~\ref{F.PredecessorGraph} showing the labeling of the jobs.} \label{F.Collisions}
\end{figure}

\begin{remark} We allow multiple labels for jobs that are immediate predecessors to more than one job (or more than one piece of the same job), since each path leading to a job represents a distinct sequence of fragments, with its own waiting time. 
\end{remark}

\begin{remark}
As pointed out earlier, the concept of the predecessor graph was introduced in \cite{Kar_Kel_Suh_94}, although in less detail than here, e.g., the labeling of jobs is not rigorous and there is no mention of multiple labels. In our case, the new coupling technique that we use, which yields a stronger mode of convergence in the main theorem and allows us to take $N$ to be random, requires the more careful treatment described above.
\end{remark}

\bigskip
\section{Analysis of the steady-state waiting time} \label{S.MainAnalysis}

To analyze the waiting time of the tagged job we now derive a high-order Lindley recursion. To this end, let $W_{\bf i}^{(n, t_0)}$ denote the waiting time of the job having label ${\bf i}$ in the network with $n$ servers and that starts empty at time $t_0$.  We also define $\hat A_0 = \{ \emptyset\}$, and $\hat A_r = \{ (i_1, \dots, i_r) \in \mathcal{G}_n(t_0)\}$ for $r \geq 1$, to be the set of labels in the predecessor graph at graph distance $r$ from the tagged job, i.e., labels whose corresponding job is connected to the tagged job by a directed path of length $r$. For ${\bf i} \in \hat A_k$ and $r \geq 1$ let 
$$\mathcal{B}_{{\bf i}, r} = \left\{ {\bf j} \in \hat A_{k+r}: {\bf j} = ({\bf i}, i_{k+1}, \dots, i_{k+r}) \right\}$$
be the set of labels at distance $r$ from ${\bf i}$, with the simplified notation $\mathcal{B}_{\bf i} \equiv \mathcal{B}_{{\bf i},1}$. 

We then have that the tagged job's waiting time is given by
\begin{equation} \label{eq:Recursion}
W^{(n, t_0)}_\emptyset = \max\left\{ 0, \, \max_{i \in \mathcal{B}_\emptyset} \left( \chi_{i} - \hat \tau_{i} + W^{(n, t_0)}_{i} \right) \right\},
\end{equation}
with the boundary condition that the first job to arrive after time $t_0$, and any other job that arrives thereafter and is the first one to use its assigned servers, will have a waiting time of zero (recall that the system is empty at time $t_0$).

To analyze \eqref{eq:Recursion} define $\hat X_\emptyset \equiv 0$ and $\hat X_{\bf i} = \chi_{{\bf i}} - \hat \tau_{{\bf i}}$ for ${\bf i} = (i_1, \dots, i_{k})$. Next, let
$$\kappa = \max\{ r \in \mathbb{N}_+: |\mathcal{B}_{\emptyset,r}| > 0 \},$$
where $|A|$ denotes the cardinality of set $A$. Note that $\kappa$ is a random variable and corresponds to the maximum length of any directed path in $\mathcal{G}_n(t_0)$. Furthermore, for any ${\bf i} \in \hat A_\kappa$ we have $W_{\bf i}^{(n, t_0)} = 0$, and therefore, for any ${\bf i} \in \hat A_{\kappa-1}$, 
$$W_{\bf i}^{(n, t_0)} = \max\left\{ 0, \max_{{\bf j} \in \mathcal{B}_{\bf i}} \hat X_{\bf j} \right\} .$$
Similarly, iterating \eqref{eq:Recursion} we obtain for ${\bf i} \in \hat A_{\kappa-2}$, 
\begin{align*}
W_{\bf i}^{(n, t_0)} &= \max\left\{ 0, \, \max_{{\bf j} \in \mathcal{B}_{\bf i}} \left( \hat X_{\bf j} + W_{\bf j}^{(n, t_0)} \right) \right\} \\
&=  \max\left\{ 0, \, \max_{{\bf j} \in \mathcal{B}_{{\bf i},1}} \hat X_{\bf j}, \, \max_{{\bf j} \in \mathcal{B}_{{\bf i},2} } \left( \hat X_{{\bf j}|1} +  \hat X_{\bf j}  \right) \right\}.
\end{align*}
In general, after iterating \eqref{eq:Recursion} $\kappa$ times we obtain
\begin{align}
W^{(n, t_0)}_\emptyset &= \max\left\{ 0, \max_{j \in \mathcal{B}_{\emptyset, 1}} \hat X_j, \dots, \max_{{\bf j} \in \mathcal{B}_{\emptyset, \kappa}} \left( \hat X_{{\bf j}| 1} + \hat X_{{\bf j} | 2} + \cdots + \hat X_{\bf j} \right) \right\} .   \label{eq:Solution}
\end{align}

Having now a recursive equation for the waiting time, we need to identify conditions under which this queueing system will be stable, and then describe the stationary distribution of the waiting time. The key idea to solve both problems is that the predecessor graph is very close to being a tree; more precisely, the only thing that prevents it from being a tree is the occasional arrival of a job that is an immediate predecessor to two or more jobs in $\mathcal{G}_n(t_0)$. It turns out that under the scaling we consider in our model (arrival rate equal to $\lambda n$), the probability of this occurring within the timeframe needed for the tagged job to start its service is very small (geometrically small). Once we show that this is the case, taking the limit as $t_0 \to -\infty$ will yield the stability of the network.

To describe the stationary distribution of the waiting time we first observe that, provided the first time that two paths in the predecessor graph merge occurs after the tagged job has initiated its service, we have that the $\hat \tau$'s will be i.i.d. exponential random variables with some rate $\lambda_n^*$ and the $\chi$'s will be i.i.d. with distribution $B$ (since they will all belong to different jobs). It follows that under the same conditions that guarantee the stability of the network we would have that, after taking the limit as $t_0 \to -\infty$ and the number of servers $n \to \infty$, $W_\emptyset^{(n, t_0)}$ would have to converge to a solution to the stochastic fixed-point equation
\begin{equation} \label{eq:HighOrderLindley}
W \stackrel{\mathcal{D}}{=} \max\left\{ 0, \, \max_{1\leq i \leq N} \left( \chi_{ i} -  \tau_{i} + W_{i} \right) \right\},
\end{equation}
where $\{W_i\}$ are i.i.d. copies of $W$, independent of $(N, \chi_1, \tau_1, \chi_2, \tau_2, \dots)$, with $N$ having distribution $f \triangleq \lim_{n \to \infty} f_n$, the $\chi$'s i.i.d. having distribution $B$, and 
 the $\tau$'s i.i.d. exponential random variables with rate $\lambda^* \triangleq \lim_{n \to \infty} \lambda_n^*$, all random variables independent of each other.   Note also that we have replaced the set over which the maximum is computed, $i \in \mathcal{B}_\emptyset$, with $1 \leq i \leq N$, since in stationarity all the pieces of a job have an immediate predecessor.

It turns out that \eqref{eq:HighOrderLindley} has multiple solutions \cite{Biggins_98}, unlike the standard Lindley equation for $N \equiv 1$. It is the structure of \eqref{eq:Solution} that will allow us to identify the correct one. As we will see in the following sections, the appropriate solution is the so-called endogenous one, which is also the minimal one in the usual stochastic order sense.

To identify the value of $\lambda^*$ recall that in the time reversed setting, for the system with $n$ servers, we can think of independent Poisson processes each generating jobs with a tag of the form $(s_1, s_2, \dots, s_k)$, for $1\leq k \leq m_n$, and $s_i \in \{1, 2, \dots, n\}$ for all $i$, at rate
$$\lambda_k = \frac{f_n(k) \lambda n}{n!/(n-k)!}.$$ 
Moreover, a piece of a job requiring service at server $s_i$ can have as an immediate predecessor any job of any size requiring service at server $s_i$. In particular, there are a total of $\binom{n-1}{k-1} k!$ possible predecessors of size $k$, and therefore, the inter arrival time between the piece of the job requiring service from server $s_i$ and its predecessor is exponentially distributed with rate
\begin{equation} \label{eq:Lambda_n}
\lambda^*_n=\sum_{k=1}^{m_n} \lambda_k \binom{n-1}{k-1} k! = \sum_{k=1}^{m_n} \frac{f_n(k) \lambda n}{\binom{n}{k}} \binom{n-1}{k-1} = \lambda \sum_{k=1}^{m_n} k f_n(k).
\end{equation}

It follows that, assuming $f_n$ is uniformly integrable, 
$$\lambda^* = \lambda \sum_{k=1}^\infty k f(k) = \lambda E[ N].$$

Equation \eqref{eq:HighOrderLindley} is known in the literature (\cite{Kar_Kel_Suh_94, Biggins_98, Jel_Olv_14}) as a high-order Lindley equation, and the behavior of its endogenous solution is given in Section \ref{SS.TheLimit}.

\subsection{Main result} \label{SS.MainResult}

Before we formulate the main result of this paper it is convenient to specify the conditions we need to impose on $f_n$, $\lambda$, and $B$. Recall that the service requirements of a job of size $k$, $(\chi^{(1)}, \dots \chi^{(k)})$, are allowed to be arbitrarily dependent, say having some joint distribution $B_k({\bf x})$ on $\mathbb{R}_+^k$, but are assumed to be independent of the number of pieces. In this notation, $B$ is the distribution of the service requirement of a randomly chosen piece, i.e.,
$$B(x) = \sum_{k=1}^{m_n} f_n(k) \sum_{i=1}^k \frac{1}{k} P(\chi^{(i)} \leq x).$$
Throughout the paper, $\Rightarrow$ denotes convergence in distribution.

\begin{assum} \label{A.Nconditions}
Suppose that $f_n$ is a distribution on $\{1 ,2, \dots, m_n\}$, with $m_n \leq n$, $B$ is a distribution on $\mathbb{R}_+$, and $\lambda > 0$.
\begin{list}{\roman{enumi})}{\itemsep 0pt \leftmargin 20pt }
\usecounter{enumi}
\item Suppose there exists a distribution $f$ on $\mathbb{N}_+$, having finite mean, such that $f_n \Rightarrow f$ as $n \to \infty$ and $f_n$ is uniformly integrable.
\item Suppose there exists $\beta > 0$ such that 
$$E\left[ \sum_{i=1}^N e^{\beta (\chi_i - \tau_i)} \right] = \frac{\lambda^*}{\lambda^* + \beta} \, E[N] E\left[ e^{\beta \chi_1 } \right]< 1,$$
where $N$ is distributed according to $f$, $\{\chi_i\}$ are i.i.d. random variables with distribution $B$, independent of $N$, and $\{ \tau_i\}$ are i.i.d. exponentially distributed random variables with rate $\lambda^* = \lambda E[N]$, independent of $(N, \chi_1, \dots, \chi_N)$. 
\item Suppose that 
$$\lim_{n \to \infty} \frac{1}{n} \sum_{k=1}^{m_n} k^2 f_n(k) = 0.$$
\end{list}
\end{assum}

To give some examples of distributions for which Assumption \ref{A.Nconditions} is satisfied, let $N$ be distributed according to $f$ and consider 
$$f_n(k) = P( \min\{ N, m_n\} = k) \qquad \text{or} \qquad f_n(k) = P(N = k| N \leq m_n).$$
 In both cases, provided $E[N] < \infty$, we can take any $m_n \to \infty$, including $m_n = n$, and have $f_n$ uniformly integrable, since the monotone convergence theorem gives $E[ \min\{N, m_n\} ] \to E[N]$ and $E[ N | N \leq m_n] = E[ N 1(N \leq m_n)]/P(N \leq m_n) \to E[N]$. 
 Finally, Assumption~\ref{A.Nconditions} (iii) would be satisfied in both examples, with $m_n = n$, if $E[N^{1+\epsilon}] < \infty$ for some $\epsilon > 0$, since then
 $$\lim_{n\to\infty} \frac{1}{n} \sum_{k=1}^{n} k^2 f_n(k) \leq \lim_{n \to \infty} \frac{ n^{1-\epsilon} }{n} \sum_{k=1}^n k^{1+\epsilon} f_n(k) = E[N^{1+\epsilon}] \lim_{n\to\infty} n^{-\epsilon} = 0. $$
In case $E[N^{1+\epsilon}] = \infty$ for all $\epsilon > 0$, one would need to take $m_n = o(n)$ to obtain
$$\lim_{n\to\infty} \frac{1}{n} \sum_{k=1}^{m_n} k^2 f_n(k) \leq \lim_{n \to \infty} \frac{m_n }{n} \sum_{k=1}^n k f_n(k) = E[N] \lim_{n\to\infty} \frac{m_n}{n} = 0.$$

We are now ready to formulate the main theorem. 

\begin{theo} \label{T.Main}
Let $W_\emptyset^{(n, t_0)}$ denote the waiting time, excluding service, of the tagged job (the first job to arrive after time zero) when we start the system empty at time $t_0 < 0$ and the network consists of $n$ servers. Suppose that
\begin{equation} \label{eq:stabilityCond}
 E[\hat N] E\left[ e^{\beta (\chi - \hat \tau)} \right] < 1
\end{equation}
for some $\beta > 0$, where $\hat N$ has distribution $f_n$, $\chi$ has distribution $B$ and $\hat \tau$ is exponentially distributed with rate $\lambda_n^*$ and is independent of $\chi$. Then, for any fixed number of servers $n$,
$$\lim_{t_0 \to - \infty} W_\emptyset^{(n, t_0)} = W^{(n)} \quad a.s.$$
for some finite random variable $W^{(n)}$. Moreover, provided Assumption \ref{A.Nconditions} is satisfied, 
$$W^{(n)} \Rightarrow W,$$
as $n \to \infty$, where $W$ is the endogenous solution to \eqref{eq:HighOrderLindley}. Furthermore, for any $p > 0$,
$$E\left[ (W^{(n)})^p \right] \to E\left[ W^p \right] < \infty, \qquad n \to \infty.$$
\end{theo}

The key idea for the proof of the stability result is to couple the predecessor graph with a weighted branching tree (\cite{Rosler_93, Jel_Olv_12a}) and show that the waiting time of the tagged job is dominated by the  maximum of the random walks along all the paths of the tree.  The identification of the limit with the endogenous solution to the high-order Lindley equation \eqref{eq:HighOrderLindley} will follow from a similar coupling argument between the predecessor graph and a weighted branching tree, in which we will show that with high enough probability the tagged job will initiate its service before we observe the first merging of paths. This critical timescale at which the first merging of paths is observed also explains why the dependence among the different service requirements of a job plays no role, since a typical job will only ``see" one fragment from each of its predecessors still present in the system when it arrives. This dependence does  however impact the sojourn time, i.e., the time a job spends in the system from the moment it arrives until all its fragments complete service, which in the limit is given by
\begin{equation} \label{eq:SojournLimit}
T = \max\left\{ 0, \max_{1 \leq i \leq N} (\chi_i - \tau_i + W_i) \right\} + \max_{1 \leq j \leq N} \chi^{(j)},
\end{equation}
where $(\chi^{(1)}, \dots, \chi^{(k)})$ is distributed according to $B_k({\bf x})$ and is independent of $(N, \{\chi_i\}, \{\tau_i\}, \{W_i\})$, and the $\{W_i\}$ are i.i.d.~copies of the endogenous solution to \eqref{eq:HighOrderLindley}.

\subsection{Analyzing the limit: Generalized Cram\'er-Lundberg approximation}  \label{SS.TheLimit}

As stated in Theorem \ref{T.Main}, the stationary waiting time in the system with $n$ servers converges to the endogenous solution to the stochastic fixed-point equation \eqref{eq:HighOrderLindley}, which receives its name since it can be explicitly constructed on a weighted branching process. For completeness, we now briefly describe the construction of a weighted branching tree, which is more general than the setup considered in this paper.

Let $(Q, N, C_1, C_2, \dots)$ be a vector with $N \in \mathbb{N} \cup \{ \infty\}$, and $Q, \{C_i\}$ real-valued; the interpretation of $Q$ and the $\{C_i\}$ depends on the application. Given a sequence of i.i.d. vectors \linebreak $\left\{ (Q_{\bf i}, N_{\bf i}, C_{({\bf i},1)}, C_{({\bf i}, 2)}, \dots) \right\}_{{\bf i} \in U}$ having the same distribution as the generic branching vector \linebreak $(Q, N, C_1, C_2, \dots)$, we use the random variables $\{ N_{\bf i} \}_{{\bf i} \in U}$ to determine the structure of a tree as follows. Let $A_0 = \left\{ \emptyset \right\}$ and
\begin{align} 
A_r &= \{ ({\bf i}, i_r) \in U:  {\bf i} \in A_{r-1}, 1 \leq i_r \leq N_{\bf i} \}, \quad r \geq 1, \label{eq:AnDef}
\end{align}
be the set of individuals in the $r$th generation. Next, assign to each node ${\bf i}$ in the tree a weight $\Pi_{\bf i}$ according to the recursion
$$\Pi_{\emptyset} = 1, \qquad \Pi_{({\bf i}, j)} = C_{({\bf i}, j)} \Pi_{\bf i}. $$
Each weight $\Pi_{\bf i}$ is also usually multiplied by its corresponding value $Q_{\bf i}$ to construct solutions to non-homogeneous stochastic fixed-point equations.

In the general formulation, the vector $(Q, N, C_1, C_2, \dots)$ is allowed to be arbitrarily dependent, although for the special case appearing in this paper we will have $N < \infty$ a.s., $Q \equiv 1$, and the $\{C_i\}$ nonnegative, i.i.d., and independent of $N$. For more details we refer the reader to \cite{Rosler_93, Jel_Olv_12a, Jel_Olv_12b}.

To make the connection between the high-order Lindley equation \eqref{eq:HighOrderLindley} and the main result in \cite{Jel_Olv_14}, let $R = e^W$, $R_i = e^{W_i}$, $Q \equiv 1$, and $C_i = e^{\chi_i - \tau_i}$ to obtain
\begin{equation} \label{eq:Maximum}
R \stackrel{\mathcal{D}}{=} Q \vee \left(\bigvee_{i=1}^N C_i R_i \right),
\end{equation}
where $x \vee y$ denotes the maximum of $x$ and $y$. We refer to \eqref{eq:Maximum} with a generic branching vector of the form $(Q, N, C_1, C_2, \dots)$ with the $\{C_i\}$ nonnegative, and the $\{R_i\}$ i.i.d. copies of $R$ independent of $(Q, N, C_1, C_2, \dots)$, as the {\em branching maximum equation}. 

It is easy to verify, as was done in \cite{Jel_Olv_14}, that the random variable 
$$R \triangleq \bigvee_{r=0}^\infty \bigvee_{{\bf j} \in A_r} \Pi_{\bf j} Q_{\bf j}$$
is a solution to \eqref{eq:Maximum}, known in the literature as the endogenous solution (\cite{Aldo_Band_05, Biggins_98}). Moreover, when $Q \geq 0$, the endogenous solution is also the minimal one in the usual stochastic order sense (see \cite{Biggins_98} and also the survey paper \cite{Aldo_Band_05}  for additional references and a wide variety of max-plus equations). Taking logarithms on both sides of \eqref{eq:Maximum} (with $Q \equiv 1$), we obtain that the endogenous solution to \eqref{eq:HighOrderLindley} is given by
\begin{equation} \label{eq:EndogenousSol}
W \triangleq \bigvee_{r=0}^\infty \bigvee_{{\bf j} \in A_r} S_{\bf j} ,
\end{equation}
where  $S_{\emptyset} = 0$, $S_{\bf j} = \log \Pi_{\bf j} = X_{{\bf j}|1} + X_{{\bf j}|2} + \dots + X_{\bf j}$ for ${\bf j} \neq \emptyset$, and $X_{\bf i} = \chi_{\bf i} - \tau_{\bf i}$.  Furthermore, it was shown in \cite{Jel_Olv_14} (see Lemma 3.1) that this endogenous solution is finite almost surely provided 
$$E\left[ \sum_{i=1}^N e^{\beta X_i} \right] < 1$$
for some $\beta > 0$, which we will refer to as the stability condition. Note that with respect to the queueing model in this paper, this stability condition implies the usual ``load condition", i.e., arrival rate divided by service rate strictly smaller than one, which in this case would be $\lambda E[N] E[\chi] = E[\chi] /E[\tau] < 1$; the two are equivalent when $E[N] = 1$. 

\begin{remark}
The stability condition guarantees that $W$, as defined by \eqref{eq:EndogenousSol}, is finite almost surely. Moreover, by Theorem 4 in \cite{Biggins_98}, the existence of $\beta > 0$ such that $E\left[ \sum_{i=1}^N e^{\beta X_i} \right] \leq 1$ is the corresponding necessary condition (since $P\left( \max_{1 \leq i \leq N} (\chi_i - \tau_i) > 0 \right) > 0$). However, we do not consider in this paper the boundary condition where $E\left[ \sum_{i=1}^N e^{\theta X_i} \right] = 1$ for some $\theta > 0$ but $E\left[ \sum_{i=1}^N e^{\beta X_i} \right] > 1$ for all $\beta \neq \theta$.  
\end{remark} 

By rewriting $W$ as 
$$W = \max\left\{ 0, \, \max_{j \in A_1} X_j, \, \max_{{\bf j} \in A_2} \left( X_{{\bf j}|1} + X_{\bf j} \right), \, \max_{{\bf j} \in A_3} \left( X_{{\bf j}|1} + X_{{\bf j}|2} + X_{\bf j} \right), \dots \right\},$$
the similarities with \eqref{eq:Solution} become apparent.  To give some additional intuition as to why \eqref{eq:EndogenousSol} is the appropriate solution, it is helpful to recall the $N \equiv 1$ case, where Lindley's equation is known to have a unique solution whenever $E[X_1] < 0$. Moreover, as mentioned in the introduction, this solution can be expressed in terms of the supremum of the random walk $S_k = X_1 + \dots + X_k$, $S_0 = 0$. A standard proof of this relation consists in iterating the recursion
$$W_{n+1} = \max\{ 0, \, X_n + W_n \}, \qquad W_0 = 0,$$
to obtain 
$$W_{n+1} = \max\left\{ 0, \, X_n, \, X_n + X_{n-1}, \dots, \, X_n + X_{n-1} +\dots + X_1 \right\} \stackrel{\mathcal{D}}{=} \max_{0 \leq k \leq n} S_k.$$
It follows by taking the limit as $n \to \infty$ on both sides that the stationary waiting time in the FCFS GI/GI/1 queue satisfies
$$W \stackrel{\mathcal{D}}{=} \max_{k \geq 0} S_k.$$

It is then to be expected that the asymptotic analysis of the waiting time in the single-server queue can also be generalized to the branching setting. This is indeed the case, as was recently shown in \cite{Jel_Olv_14}. There, for the endogenous solution to the general branching maximum recursion  \eqref{eq:Maximum}, it was shown that under the root condition $E\left[ \sum_{i=1}^N C_i^\theta \right] = 1$ and the derivative condition $0 < E\left[ \sum_{i=1}^N C_i^{\theta} \log C_i \right] < \infty$, we have that
\begin{equation} \label{eq:PowerLawTail}
P( R > x) \sim H x^{-\theta}, \qquad x \to \infty,
\end{equation}
for some constants $\theta, H > 0$. Note that for the high-order Lindley's equation \eqref{eq:HighOrderLindley}, this condition translates into the existence of a root $\theta > 0$ such that $E\left[ \sum_{i=1}^N e^{\theta X_i} \right] = 1$. The power-law asymptotics of $R$ are a consequence of the Implicit Renewal Theorem on Trees from \cite{Jel_Olv_12a, Jel_Olv_12b}, which constitutes a powerful tool for the analysis of many different types of branching recursions, e.g., the maximum recursion (\cite{Alsm_Rosl_08, Jel_Olv_14}), the linear recursion or smoothing transform (\cite{Als_Big_Mei_10, Alsm_Mein_10a, Alsm_Mein_10b, Jel_Olv_12a, Jel_Olv_12b, Alsm_Dam_Ment_13}), the discounted tree sum (\cite{Aldo_Band_05}), etc.  This theorem is in turn a generalization of the Implicit Renewal Theorem of \cite{Goldie_91} for non-branching recursions, which can be used to analyze the random coefficient autoregressive process of order one and the reflected random walk, among others. The name ``implicit" refers to the fact that the Renewal Theorem is applied to a random variable $R$ (e.g., the solution to a stochastic fixed-point equation) without having knowledge of its distribution, which in turn leads to the resulting constant $H$ in the asymptotics to be implicitly defined in terms of $R$ itself. 

We conclude this section with the theorem describing the asymptotic behavior of $W$, the endogenous solution to \eqref{eq:HighOrderLindley}.

\begin{theo} \label{T.CramerLundberg}
Let $W$ be given by \eqref{eq:EndogenousSol}, with $N$ distributed according to $f$, $\{\chi, \chi_i\}$ i.i.d.~with common distribution $B$, and $\{\tau_i\}$ i.i.d.~exponentially distributed with rate $\lambda^*$; all random variables independent of each other. Suppose that for some $\theta > 0$, $E[N] E[e^{\theta \chi} ] \lambda^*/(\lambda^* + \theta)  = 1$ and $0 < E\left[ e^{\theta \chi} \chi \right] - E\left[ e^{\theta \chi} \right]/(\lambda^*+ \theta)  < \infty$. In addition, assume that for some $\epsilon > 0$, $E\left[ N^{\theta \vee (1+\epsilon)} \right] < \infty$. Then, 
$$P( W > x) \sim H e^{-\theta x}, \qquad x \to \infty,$$
where $0 < H < \infty$ is given by
$$H = \frac{(\lambda^*+\theta)^2}{\theta \lambda^* E[N]} \cdot \frac{E \left[ 1 \vee \bigvee_{i=1}^N e^{\theta ( \chi_i - \tau_i + W_i)} - \sum_{i=1}^N e^{\theta(\chi_i - \tau_i + W_i)} \right] }{ (\lambda^*+\theta) E\left[ e^{\theta \chi} \chi \right] - E\left[ e^{\theta \chi} \right] },$$
with the $\{W_i\}$ i.i.d.~copies of $W$, independent of $(N, \chi_1, \tau_1, \dots, \chi_N, \tau_N)$.  
\end{theo}

The constant $H$ in the asymptotic tail of $W$ can be computed via simulation, for example, by using the algorithm recently developed in \cite{Chen_Olv_15}, which can be used to generate the $\{W_i\}$ appearing in the expectation.

\bigskip
\section{Numerical experiments} \label{S.Numerics}

In this section of the paper we provide some numerical experiments comparing our model to two other: a non-distributed multiserver system and a split-merge queue. All the results in this section were obtained using discrete-event simulation, starting with an empty system. 

Throughout this section, we refer to the model studied in this paper as the synchronize at the beginning (SyncB) model. We also consider a split-merge queue (Split-Merge) where incoming jobs are split upon arrival into a number of pieces, and then assigned to randomly selected servers, each of which operates in a FCFS basis.  To explain how the synchronization occurs, assume that a job fragment that has completed service does not leave the server until all other pieces of the same job are done as well (since there are no output buffers in the system). Unlike in the SyncB model, fragments that have reached the front of their queues and find their servers available can begin processing immediately.

\begin{table}[t]
\begin{center} \begin{tabular}[h]{| c | c | c | l | c | c |}
\hline
\multicolumn{2}{| c |}{Job size} & Arrival & & Mean & \\
\cline{1-2} $E[N]$ & $\beta$ & rate $(\lambda n)$ & Model & Sojourn Time & 95\% C.I.\\
\hline
2 & 2/3   & 100 & SyncB   & 0.7290   & [0.6843, 0.7737]  \\
&  & & Split-Merge & 0.7389       &  [0.6943, 0.7836] \\
& & & M/G/$n$  & 0.9994         & [0.9206 , 1.0782]  \\ \hline 
10 & 6  & 6.5 & SyncB    & 1.1201   &    [1.0654, 1.1748] \\
& & & Split-Merge  & 1.1116   &    [1.0599, 1.1634]  \\
& & & M/G/$n$ & 4.9916  &     [4.6585, 5.3247]  \\ \hline
100 & 66 & 0.06 & SyncB   & 1.0138   &       [0.9995 , 1.0281]  \\
& &  & Split-Merge  & 1.0140   &       [0.9995 , 1.0284]  \\
& &  & M/G/$n$ & 49.7684     &          [46.8525, 52.6844]  \\  \hline
\end{tabular}
\vspace{5mm}
 \end{center}
\caption{Mean sojourn time. Simulated results for the average sojourn time of jobs in the SyncB, Split-Merge and M/G/$n$ models; in all models the number of servers is $n = 1000$, the arrival rate of jobs is $\lambda n$, indicated by the table, the service requirements of the pieces of a job are i.i.d.~U$(0,1)$, and the number of pieces is computed as $\hat N = N \wedge n$, with $N-1$ a mixed Poisson random variable with Pareto$(\alpha, \beta)$ rate, $\alpha = 3$, and $\beta$ according to the table. For the M/G/$n$ queue the service requirement of a job is the sum of the requirements of its pieces. All three simulations were run using the same arrivals and jobs. Simulations were run for a total of 30,000 jobs.}  \label{T.CompareThree}
\end{table}

Our first set of results compares the SyncB, Split-Merge and M/G/$n$ models, all run with the same Poisson arrivals and job distributions. For the M/G/$n$ model the service distribution is that of the sum of all the pieces in a job. The purpose of including the M/G/$n$ queue in this comparison is to illustrate the gain attained by distributing the work among parallel servers, which as Table~\ref{T.CompareThree} shows, outweighs the loss of efficiency due to the blocking. In all the experiments, we focus on the sojourn time of jobs, i.e., the amount of time a job spends in the system, from the time it arrives to the time it completes its service and leaves. Table~\ref{T.CompareThree} shows simulated values for the expected stationary sojourn time in a network with $n$ servers along with 95\% approximate confidence intervals (the parameters lie within the theoretical stable region for both the SyncB and M/G/$n$ models; we do not have a criterion for the stability of the Split-Merge model but the simulated results are consistent with stationarity). As we can see from the table, the two distributed systems are comparable, and better than the non-distributed M/G/$n$ queue. We point out that for the M/G/$n$ queue it can be verified that under the scaling considered here, the waiting time converges to zero, which reduces the sojourn time to essentially the service time of a job, i.e., a quantity of the form $\sum_{i=1}^{\hat N} \chi^{(i)}$, whereas in the SyncB model the waiting time is non-zero but the service time is $\bigvee_{i=1}^{\hat N} \chi^{(i)}$. In our experiments we have used i.i.d.~uniform service times for the job fragments and a heavy-tailed mixed Poisson for the number of pieces, in which case $\sum_{i=1}^{\hat N} \chi^{(i)}$ has a power-law tail whereas $\bigvee_{i=1}^{\hat N} \chi^{(i)}$ is bounded, which leads to a considerable number of jobs experiencing very long sojourn times in the non-distributed system.

\begin{figure}[t]
\begin{picture}(600,200)
\put(15,10){\includegraphics[scale=0.45, bb = 0 0 460 370, clip]{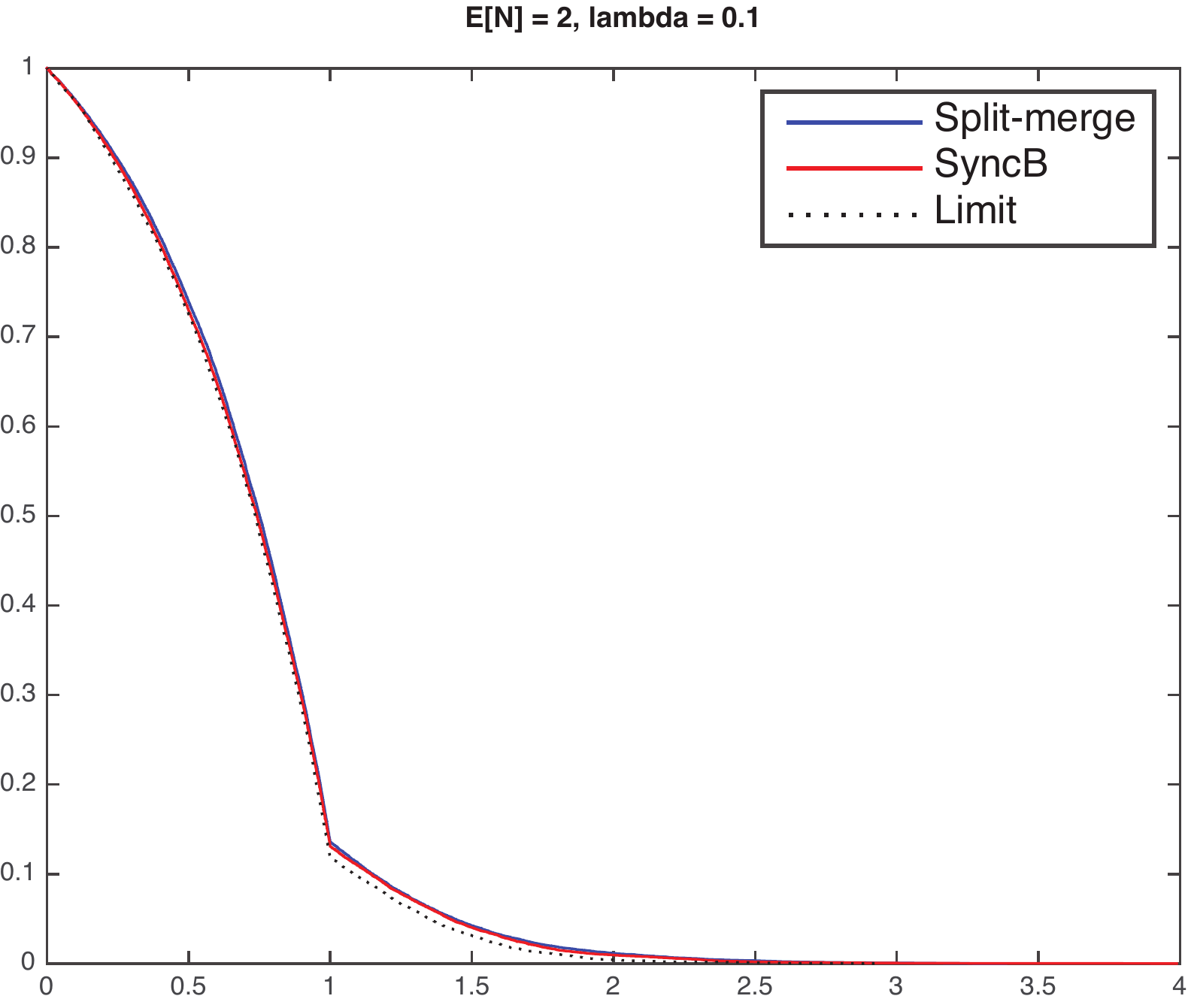}}
\put(260,10){\includegraphics[scale=0.45, bb = 0 0 460 370, clip]{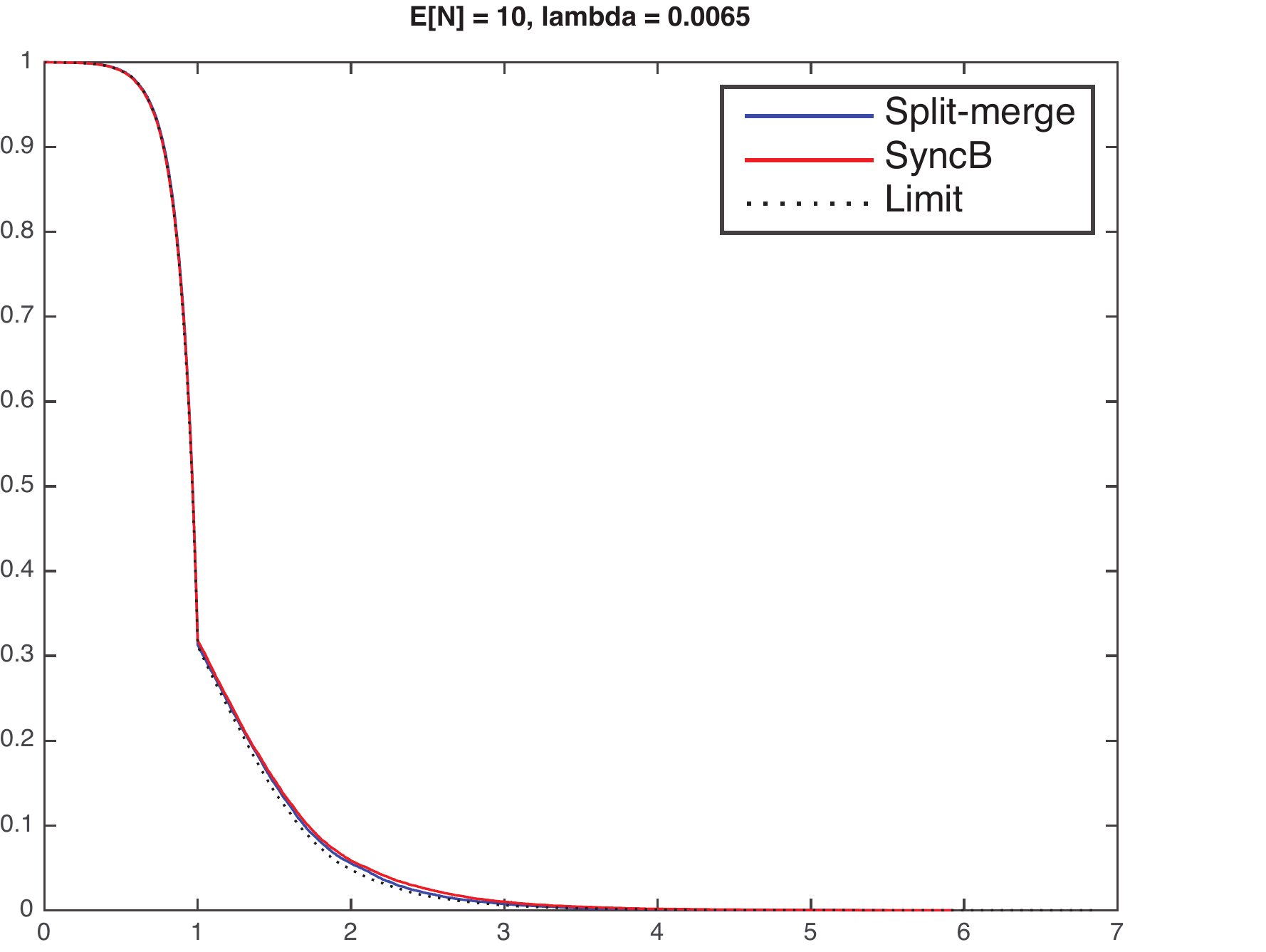}}
\put(50,193){\small \bf Sojourn time tail distribution}
\put(65,180){\footnotesize $E[N] = 2$, $\beta = 2/3$, $\lambda n = 100$}
\put(295,193){\small \bf Sojourn time tail distribution}
\put(310,180){\footnotesize $E[N] = 10$, $\beta = 6$, $\lambda n = 6.5$}
\put(0,60){\rotatebox{90}{\scriptsize P(Sojourn time $> x$)}}
\put(245,60){\rotatebox{90}{\scriptsize P(Sojourn time $> x$)}}
\put(110,0){\footnotesize $x$}
\put(360,0){\footnotesize $x$}
\end{picture}
\caption{Sojourn time tail distribution. Simulated values of the tail distributions for the SyncB and Split-Merge models. In both cases, the number of servers is $n = 1000$, the arrival rate of jobs is $\lambda n$, indicated on the plots, the service requirements of the pieces of a job are i.i.d.~U$(0,1)$, and the number of pieces is computed as $\hat N = N \wedge n$, with $N-1$ a mixed Poisson random variable with Pareto$(\alpha, \beta)$ rate, $\alpha = 3$, and $\beta$ indicated on the plots. Both simulations were run using the same arrivals and jobs for a total of 30,000 jobs. The tail distribution of the limiting sojourn time $T$ is provided for comparison.} \label{F.SojournTail}
\end{figure}
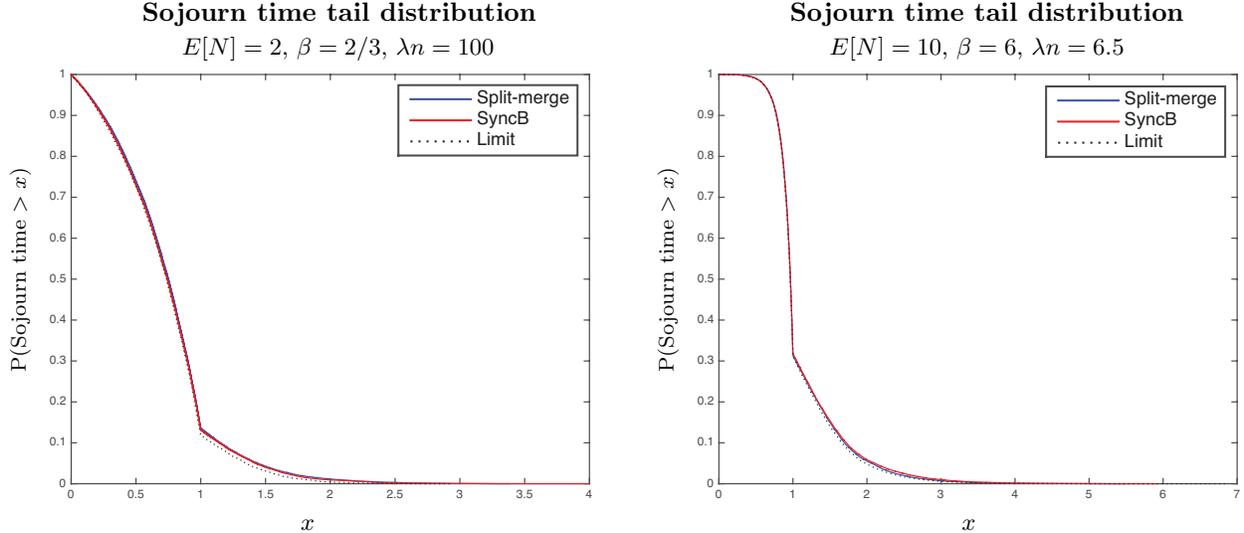

Our second set of results compare the tail distributions of the sojourn times in the SyncB and Split-Merge models. Figure~\ref{F.SojournTail} depicts two comparisons, one for $E[N] = 2$ and one for $E[N] = 10$. The parameters used in both plots are within the stability region for the SyncB model. As can be seen from the figure, the distributions of the SyncB and Split-Merge models are undistinguishable, which strongly supports the use of the SyncB model for approximating the mathematically intractable Split-Merge model. In other words, SyncB seems to provide an accurate model for MapReduce with random routing and FCFS scheduling, under light loads. We have also included the tail distribution of the limiting sojourn time \eqref{eq:SojournLimit}, computed using the algorithm in \cite{Chen_Olv_15}. The approximation works best for small values of $E[N]$ and light loads (i.e., $\lambda E[N] E[\chi]$ small).

\begin{figure}[h]
\begin{center}
\begin{picture}(200,200)
\put(10,5){\includegraphics[scale = 0.5, bb = 0 0 430 360, clip]{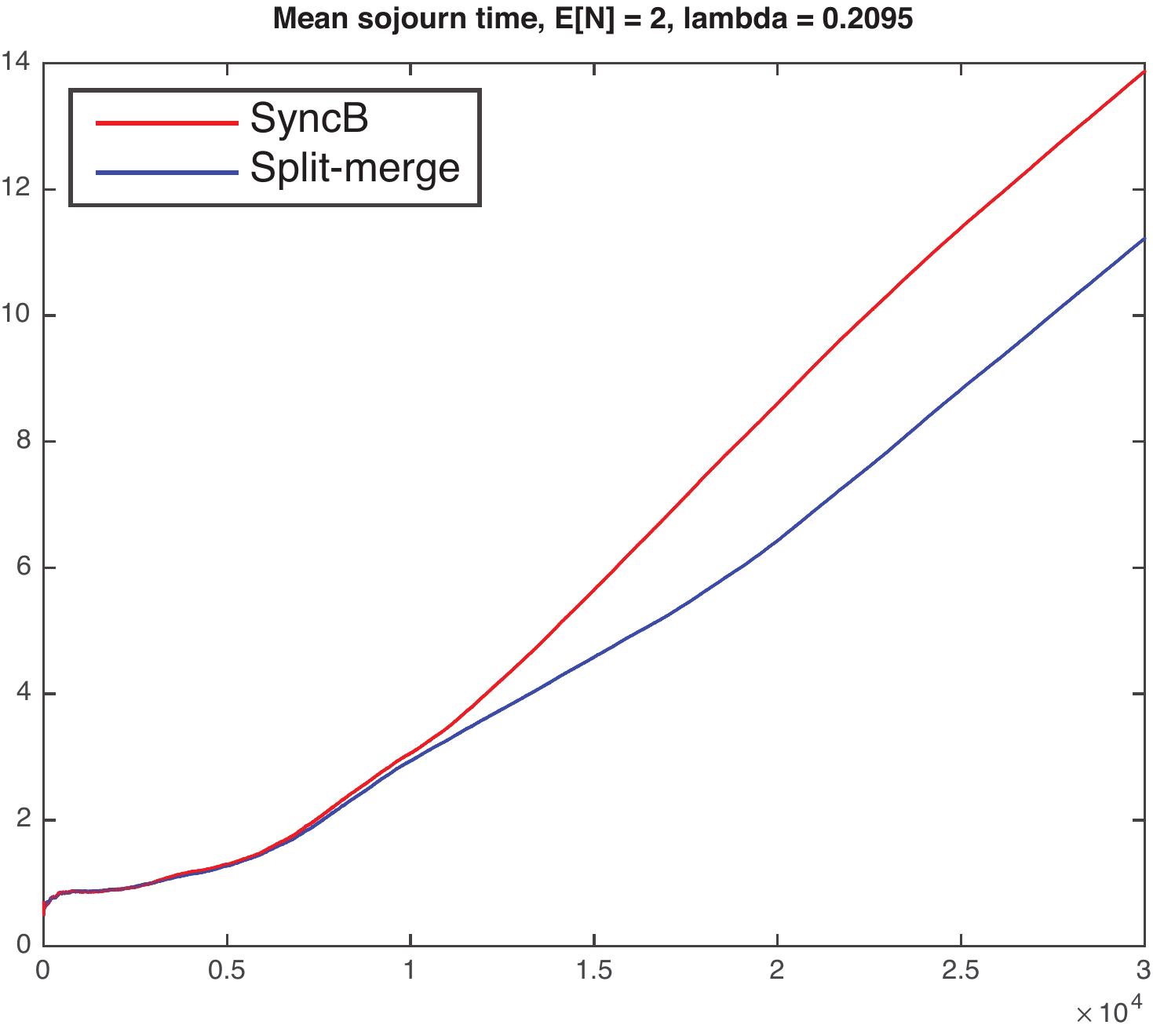}}
\put(15,195){\small \bf Mean sojourn time at stability boundary}
\put(85,0){\footnotesize Number of jobs $(t)$}
\end{picture}
\caption{Mean sojourn time. Simulated values of the average sojourn time of the first $t$ jobs for the SyncB and Split-Merge models. In both cases, the number of servers is $n = 1000$,  the arrival rate of jobs is $\lambda n$, with $\lambda = 0.2095$, the service requirements of the pieces of a job are i.i.d.~U$(0,1)$,  and the number of pieces is computed as $\hat N = N \wedge n$, with $N-1$ a mixed Poisson random variable with Pareto$(\alpha, \beta)$ rate, $\alpha = 3$, and $\beta = 2/3$ ($E[N] = 2$). The parameters are such that the SyncB model is at its theoretical boundary of stability. Both simulations were run using the same arrivals and jobs for a total of 30,000 jobs.} \label{F.Unstable}
\end{center}
\end{figure}

The last numerical result in the paper compares the SyncB and Split-Merge models right at the theoretical stability boundary of the SyncB model. Figure \ref{F.Unstable} plots the running average of the sojourn times of the first $t$ jobs, for up to 30,000 jobs. The plot is consistent with both models being unstable, and the important insight we obtain is that the stability regions for the two models seem to be very close, if not the same. This observation provides further support for the use of SyncB as a qualitatively good approximation for Split-Merge.

\bigskip
\section{Concluding remarks} \label{S.Conclusions}

The model presented in this paper captures the complexity of a large network of parallel servers with blocking and synchronization constraints. Although not an exact model for any of today's existing distributed computing algorithms, our model provides a good approximation for split-merge queues, which can be used to model the popular MapReduce algorithm under random routing and FCFS scheduling. More importantly, our model is analytically tractable, and hence provides a valuable benchmark for studying other distributed queueing models. In particular, one can think of the random routing in our model (i.e., the way in which job pieces are assigned to randomly selected servers) as a ``blind" scheduling rule, since it requires no information about the availability or the workload of any of the servers. On the other extreme, the model in \cite{Green_80} corresponds to the ``optimal" scheduling rule, where job fragments are assigned to the servers with the smallest workloads. It follows that the ``blind" and ``optimal" models can be used in conjunction to provide a good cost-benefit analysis of any other scheduling rule (e.g., send to idle servers first, or select twice as many servers as needed and choose those with the smallest workloads/number of jobs). In other words, our model can be used to price server information.  Furthermore, our model yields valuable insights about general split-merge queues, which are important since they provide upper bounds for various performance measures in today's popular fork-join networks. In particular, the unusual and rather stringent stability condition of our model highlights how critical the scheduling discipline is, since the ``optimal" model in \cite{Green_80} is stable under the usual load condition (arrival rate divided by service rate strictly smaller than one). Hence, the analysis in this paper motivates the search for  ``easily implementable" scheduling rules that can recover the weaker stability condition.

\section{Proofs} \label{S.Proofs}

This section  contains the proofs of Theorem \ref{T.Main} and Theorem \ref{T.CramerLundberg}.  To ease the exposition we separate Theorem \ref{T.Main} into two parts, the first one concerning the existence of a stationary waiting time for a fixed number of servers and a fixed arrival rate of jobs, the second one establishing the limiting distribution of the stationary waiting time as the number of servers and the arrival rate of jobs grow to infinity. 

We start by summarizing some of the notation that will be used throughout this section, starting with all the random variables involved in the predecessor graph. Let 
$$\hat U_0 = 0, \qquad \hat U_r = \max_{{\bf j} \in \hat A_r} \, \hat S_{\bf j}, \quad r \geq 1,$$
where $\hat S_{\bf j} = \hat X_{{\bf j}|1} + \hat X_{{\bf j} |2} + \dots + \hat X_{\bf j}$, $\hat X_{\bf i} = \chi_{\bf i} - \hat \tau_{\bf i}$, and $\hat A_r$ is the set of labels (not jobs) in the predecessor graph at graph distance $r$ from the tagged job. Note that $|\mathcal{B}_{\emptyset,r}| \leq |\hat A_r|$, since $\mathcal{B}_{\emptyset,r}$ refers to the set of labels in $\mathcal{G}_n(t_0)$, where there are some jobs/fragments that do not have any predecessors, and as $t_0 \to - \infty$ all jobs/fragments will have have one. We also point out that since every time multiple paths in the graph merge (i.e., every time a job that is an immediate predecessor to multiple fragments arrives) all the jobs from that point onwards will have multiple labels, then some of the $\{ \hat N_{\bf i} \}$ will be repeated, and are therefore not independent.  Similarly, the $\{ \hat \tau_{\bf i}\}$ correspond to the inter arrival times between jobs in the predecessor graph (the length of the edges), and are therefore, in general, neither independent of each other nor of the $\{\hat N_{\bf i} \}$. More precisely, the marginal distribution of each of the $\hat \tau_{\bf i}$ is exponential with rate $\lambda_n^*$, but conditionally on knowing that ${\bf i}$ shares a predecessor with one or more other jobs, its rate changes and all the inter arrival times corresponding to edges that merge into the same job become dependent. Finally, the $\{\chi_{\bf i}\}$ are identically distributed with marginal distribution $B$, and their dependance with the $\{ \hat \tau_i\}$ and $\{\hat N_{\bf i}\}$ is limited to the multiplicity of the labels (i.e., labels referring to edges that lie on merged paths have the same service requirements). Nonetheless, service requirements of the form $\chi_{\bf i}$ and $\chi_{\bf j}$ with ${\bf i} \neq {\bf j}$ may also be dependent if they correspond to fragments of the same job.  We will identify labels belonging to the same job in the predecessor graph through the equivalence relation ${\bf i} \sim {\bf j}$.

The analysis of the predecessor graph will become tractable once we identify suitable approximations where the merging of paths due to common predecessors does not occur, in other words, where the predecessor graph is truly a tree. Since these approximations will be used several times in the proofs it is convenient to define them upfront. 

In general, we will use $\tilde N$ to refer to a random variable having distribution $f_n$, $N$ to refer to a random variable having distribution $f$, $\chi$ to refer to a random variable having distribution $B$, $\tilde \tau$ to denote an exponential random variable with rate $\lambda_n^*$, and $\tau$ to denote an exponential random variable with rate $\lambda^*$.

Let $\{ \tilde N_{\bf i} \}$, denote a sequence of i.i.d. copies of $\tilde N$, $\{ \chi_{\bf i}\}$ an i.i.d. sequence of copies of $\chi$, and $\{ \tilde \tau_{\bf i} \}_{{\bf i} \in U}$ an i.i.d. sequence of copies of $\tilde \tau$, all independent of each other. Use the $\{ \tilde N_{\bf i}\}$ to define a branching process by setting $\tilde A_0 = \{ \emptyset\}$ and $\tilde A_r = \{ ({\bf i}, i_r): {\bf i} \in \tilde A_{r-1}, \, 1\leq i_r \leq \tilde N_{\bf i}\}$ for $r \geq 1$. Next, set $\tilde X_{\bf i} = \chi_{\bf i} - \tilde \tau_{\bf i}$, $\tilde S_{\bf j} = \tilde X_{{\bf j}|1} + \tilde X_{{\bf j} |2} + \dots + \tilde X_{\bf j}$ and define
$$\tilde U_0 = 0, \qquad \tilde U_r = \max_{{\bf j} \in \hat A_r} \, \tilde S_{\bf j}, \quad r \geq 1.$$

Similarly, by repeating the construction given above after removing the $^\sim$ from all the random variables we obtain
$$U_0 = 0, \qquad  U_r = \max_{{\bf j} \in A_r} \, S_{\bf j}, \quad r \geq 1,$$
where  $S_{\bf j} = X_{{\bf j}|1} + X_{{\bf j} |2} + \dots +  X_{\bf j}$ and $X_{\bf i} = \chi_{\bf i} - \tau_{\bf i}$. We point out that if $f_n = f$ for all $n \geq n_0$, then there is no difference between all the $^\sim$ random variables and those without it. 

We are now ready to prove the first part of Theorem \ref{T.Main}.

\bigskip

\begin{proof}[Proof of Theorem \ref{T.Main} {\bf (Stability)}]
We need to show that provided $E[\hat N] E\left[ e^{\beta (\chi - \hat \tau)} \right] < 1$ the limit $\lim_{t_0 \to -\infty} W_\emptyset^{(n,t_0)}$ exists and is finite a.s. To this end, note that since
\begin{equation} \label{eq:AlltimeMax}
W_\emptyset^{(n,t_0)} = \bigvee_{r=0}^{\kappa} \bigvee_{{\bf j} \in \mathcal{B}_{\emptyset,r}} \left( \hat X_{{\bf j}|1} + \hat X_{{\bf j}|2} + \dots + \hat X_{\bf j} \right),
\end{equation}
and as $t_0 \to -\infty$ we have that $\kappa \to \infty$ a.s. and $\mathcal{B}_{\emptyset,r} \uparrow \hat A_r$, it follows by monotone convergence that
$$\lim_{t_0 \to -\infty} W_{\emptyset}^{(n, t_0)}  = \bigvee_{r=0}^\infty \hat U_r  \triangleq W^{(n)} \qquad \text{a.s.}$$
Therefore, it only remains to verify that $W^{(n)} < \infty$ a.s. 

To establish the finiteness of $W^{(n)}$ we note that it suffices to show that
$$P( \hat U_r > 0 \text{ i.o.} ) = 0.$$ 
This in turn will follow from the Borel-Cantelli Lemma once we show that
\begin{equation} \label{eq:GeometricBound}
P(\hat U_r > 0 ) \leq c^r
\end{equation}
for some constant $0 < c < 1$. Therefore, we focus on showing \eqref{eq:GeometricBound}.

Recall from the observations made at the beginning of this section that the $\{\hat \tau_{\bf i} \}_{{\bf i} \in \mathcal{G}_n(t_0)}$ are neither i.i.d. nor independent of the $\{\hat N_{\bf i}\}_{{\bf i} \in \mathcal{G}_n(t_0)}$. More precisely, recall that for each piece of a job requiring service at server $s_i$ there are $\binom{n-1}{k-1} k!$ possible immediate predecessors of size $k$ (i.e., jobs that also require service from server $s_i$). Since the arrival of jobs is assumed to follow a Poisson process, this leads in \eqref{eq:Lambda_n} to the inter arrival time between a fixed piece of a job and its unique immediate predecessor to be exponentially distributed with rate $\lambda_n^*$. The problem arises when the piece of a job has two or more immediate predecessors, in which case the rate for the exponential changes. 

Consider an arrival that is predecessor to two jobs (or two pieces of the same job), $j_1$ and $j_2$, and note that there must be two different servers, say $s_{i_1}$ and $s_{i_2}$, that are required by the arriving job and that are also assigned to jobs $j_1$ and $j_2$, respectively.  There are only $\binom{n-2}{k-2} k!$ possible jobs of size $k$ requiring service by servers $s_{i_1}$ and $s_{i_2}$, and therefore, the rate at which such a predecessor arrives is given by 
$$\lambda_n^{(2)} =\sum_{k=2}^{m_n} \lambda_k \binom{n-2}{k-2} k!.$$
In general, a job that is predecessor to jobs $j_1, j_2, \dots, j_r$ in the graph arrives at a rate
$$\lambda_n^{(r)} = \sum_{k=r}^{m_n} \lambda_k \binom{n-r}{k-r} k! \leq \sum_{k=1}^{m_n} \lambda_k \binom{n-1}{k-1} k! = \lambda_n^*.$$

As for the lack of independence between the $\{ \hat \tau_{\bf i}\}$, note that the inter arrival times between pieces of jobs that have a common immediate predecessor are dependent. The sequence $\{ \hat \tau_{\bf i}\}$ is also dependent on the $\{\hat N_{\bf i}\}$, since a large number of jobs awaiting for a predecessor to arrive increases the probability of an arriving job being predecessor to two or more pieces at a time. Hence, the analysis of $P(\hat U_r > 0)$ needs some care. 

We start by using Markov's inequality to obtain
\begin{align*}
P\left( \hat U_r > 0 \right) &= P\left( \max_{{\bf j} \in \hat A_r}  \hat S_{\bf j} > 0 \right) \leq E\left[ \bigvee_{{\bf j} \in \hat A_r} e^{\beta \hat S_{\bf j}}  \right] \leq E\left[ \sum_{{\bf j} \in \hat A_r} e^{\beta \hat S_{\bf j} } \right] .
\end{align*}
Now rewrite the last expectation as follows
\begin{align*}
E\left[ \sum_{{\bf j} \in \hat A_r} e^{\beta \hat S_{\bf j} } \right] &= \sum_{{\bf j} \in \mathbb{N}_+^r} E\left[  e^{\beta S_{\bf j} } 1( {\bf j} \in \hat A_r) \right],
\end{align*}
and notice that
$$1({\bf j} \in \hat A_r)  = \prod_{k=0}^{r-1} 1( j_{k+1} \leq \hat N_{{\bf j}|k}),$$
and is therefore independent of the $\{ \hat \tau_{{\bf j}|k} \}_{k=1}^{r}$. Since all labels along a path correspond to different jobs, then the vectors $\{ (\hat N_{{\bf j}|k}, \chi_{({\bf j}|k, 1)}, \dots, \chi_{({\bf j}|k, \hat N_{{\bf j}|k})} ) \}_{k=0}^{r-1}$ are i.i.d. copies of $(\tilde N, \chi_1, \dots, \chi_{\tilde N})$, where the $\{\chi_i\}$ are i.i.d. copies of $\chi$, independent of $\tilde N$. To eliminate the dependence between this last sequence and the inter arrival times note that we can replace the $\{ \hat \tau_{{\bf j}|k} \}_{k=1}^{r}$ with i.i.d. copies of $\tilde \tau$, independent of $\{ (\hat N_{{\bf j}|k}, \chi_{({\bf j}|k, 1)}, \dots, \chi_{({\bf j}|k, \hat N_{{\bf j}|k})} ) \}_{k=0}^{r-1}$ to obtain
$$e^{\beta \hat S_{\bf j} } 1( {\bf j} \in \hat A_r) \leq_{\text{s.t.}} e^{\beta \tilde S_{\bf j} } 1({\bf j} \in \tilde A_r),$$
where $\leq_{\text{s.t.}}$ denotes the usual stochastic order, and the $\{ \tilde S_{\bf i}\}$ were defined at the beginning of this section. It follows that 
\begin{equation} \label{eq:ExpectationBound} 
P\left( \hat U_r > 0 \right) \leq E\left[ \sum_{{\bf j} \in \hat A_r} e^{\beta \hat S_{\bf j} } \right] \leq E\left[ \sum_{{\bf j} \in \tilde A_r} e^{\beta \tilde S_{\bf j} } \right],
\end{equation}
where the last expectation can be computed using standard weighted branching processes arguments (see, e.g., \cite{Jel_Olv_12a}) and is given by
\begin{equation} \label{eq:BranchingExpectation}
E\left[ \sum_{{\bf j} \in \tilde A_r} e^{\beta \tilde S_{\bf j} } \right] = \left( E\left[ \sum_{i=1}^{\tilde N} e^{\beta (\chi_i - \tilde \tau_i)} \right] \right)^r = \left( E[\tilde N] E\left[ e^{\beta(\chi - \tilde \tau)} \right] \right)^r.
\end{equation}
Setting $c = E[\tilde N] E\left[ e^{\beta(\chi - \tilde \tau)} \right] < 1$ completes the proof. 
\end{proof}

\bigskip

For the second part of the main theorem we will prove that 
$$W^{(n)} \stackrel{W_p}{\longrightarrow} W \qquad \text{as } n \to \infty$$
for any $p \geq 1$, where $W_p$ denotes the Wasserstein distance of order $p$ (see, e.g., \cite{Villani_2009}, Chapter 6). This is equivalent to convergence in distribution plus convergence of all the moments of order up to $p$ (see Theorem 6.8 in \cite{Villani_2009}).

To this end, we will consider three different sets of processes that will yield intermediate approximations between $W^{(n)}$ and $W$. In particular, we will show that if $\mu_n$ is the probability measure of $W^{(n)}$, $\hat \nu_k$ is the probability measure of $\bigvee_{r=0}^k \hat U_r$, $\tilde \nu_k$ is that of $\bigvee_{r=0}^k \tilde U_r$, $\nu_k$ of $\bigvee_{r=0}^k U_r$, and, finally, $\mu$ is the probability measure of 
$$W = \bigvee_{r=0}^\infty U_r,$$
then
\begin{equation} \label{eq:TriangleIneq}
W_p(\mu_n, \mu) \leq W_p(\mu_n, \hat \nu_{r_n}) + W_p(\hat \nu_{r_n}, \tilde \nu_{r_n}) + W_p(\tilde \nu_{r_n}, \nu_{r_n}) + W_p(\nu_{r_n}, \mu) \to 0
\end{equation}
as $n \to \infty$, for some $r_n \to \infty$. 

The technical difficulty in the proofs lies in the need to construct explicit couplings of the pairs of probability measures involved for which we can show that their difference converges to zero in $L_p$ norm. We point out that although $W^{(n)}$ and $\bigvee_{r=0}^{r_n} \hat U_r$, as well as $W$ and $\bigvee_{r=0}^{r_n} U_r$, are naturally defined on the same probability space, all other pairs are not. The proof of the many servers limit part of Theorem \ref{T.Main} is based on a series of results.

\begin{lemma} \label{L.LpConvergence}
Suppose Assumption \ref{A.Nconditions} (i)-(ii) is satisfied.  Then, for any $r_n \to \infty$ as $n \to \infty$, and any $p \geq 1$, we have that
$$\lim_{n \to \infty} E\left[ \left| W^{(n)} - \bigvee_{r=0}^{r_n} \hat U_r \right|^p \right] = 0 \qquad \text{and} \qquad \lim_{n \to \infty} E\left[ \left| W - \bigvee_{r=0}^{r_n} U_r \right|^p \right] = 0.$$
In particular, this implies that, as $n \to \infty$, 
$$W_p(\mu_n, \hat \nu_{r_n}) \to 0 \qquad \text{and} \qquad W_p(\nu_{r_n}, \mu) \to 0.$$
\end{lemma}

\begin{proof}
Let $\beta > 0$ be the one from Assumption \ref{A.Nconditions} (ii) and let $\rho_\beta = E[N] E\left[ e^{\beta (\chi - \tau)} \right] < 1$. Fix $0 < \epsilon < 1 - \rho_\beta$ and note that
\begin{align*}
E[\tilde N ] E\left[ e^{\beta (\chi - \tilde \tau)} \right] &= E[\tilde N] E\left[ e^{\beta \chi } \right] \left( \frac{\lambda_n^*}{\lambda_n^* + \beta} \right) .
\end{align*}
By Assumption \ref{A.Nconditions} (i) we have that $E[\tilde N] \to E[N]$, and therefore $\lambda_n^* \to \lambda^*$ as $n \to \infty$. Hence, 
\begin{equation} \label{eq:LimitingRho}
\lim_{n \to \infty} E[\tilde N] E\left[ e^{\beta (\chi - \tilde \tau)} \right] = \rho_\beta.
\end{equation}
It follows that for large enough $n$, 
$$E[\tilde N ] E\left[ e^{\beta (\chi - \tilde \tau)} \right]  \leq \rho_\beta + \epsilon < 1.$$

Next, note that
\begin{align*}
E\left[ \left| W^{(n)} - \bigvee_{r=0}^{r_n} \hat U_r \right|^p \right] &= E\left[ \left( \left( \bigvee_{r=r_n+1}^\infty \hat U_r -  \bigvee_{r=0}^{r_n} \hat U_r \right)^+ \right)^p \right] \leq E\left[ \left( \left( \bigvee_{r=r_n+1}^\infty \hat U_r \right)^+ \right)^p \right] \\
&= E\left[  \bigvee_{r=r_n+1}^\infty \left( \hat U_r^+  \right)^p \right] \leq \sum_{r=r_n+1}^\infty E\left[ \left( \hat U_r^+  \right)^p \right].
\end{align*}
To analyze the last expectation note that by Markov's inequality,
\begin{align*}
E\left[ \left( \hat U_r^+  \right)^p \right] &= \int_0^\infty P\left( \left( \hat U_r^+  \right)^p > x \right) dx = \int_0^\infty P\left( e^{\beta \hat U_r} > e^{\beta x^{1/p}} \right) dx \\
&\leq E\left[ e^{\beta \hat U_r} \right] \int_0^\infty e^{-\beta x^{1/p}} dx = E\left[ \bigvee_{{\bf j} \in \hat A_r} e^{\beta \hat S_{\bf j}} \right] \frac{p}{\beta^p} \int_0^\infty u^{p-1} e^{-u} du,
\end{align*}
where $\int_0^\infty u^{p-1} e^{-u} du = E[\xi^{p-1}] < \infty$ with $\xi$ exponentially distributed with rate one. Letting $C_{\beta,p} = p E[\xi^{p-1}]/\beta^p$ gives
$$E\left[ \left( \hat U_r^+  \right)^p \right] \leq C_{\beta,p} \, E\left[ \bigvee_{{\bf j} \in \hat A_r} e^{\beta \hat S_{\bf j}} \right] \leq C_{\beta,p} \, E\left[ \sum_{{\bf j} \in \hat A_r} e^{\beta \hat S_{\bf j}} \right].$$
Moreover, as shown in the proof of the stability part of Theorem \ref{T.Main} (see \eqref{eq:ExpectationBound} and \eqref{eq:BranchingExpectation}), we have
$$E\left[ \sum_{{\bf j} \in \hat A_r} e^{\beta \hat S_{\bf j}} \right] \leq  E\left[ \sum_{{\bf j} \in \tilde A_r} e^{\beta \tilde S_{\bf j}} \right] =  \left( E[\tilde N ] E\left[ e^{\beta (\chi - \tilde \tau)} \right]   \right)^r.$$ 
It follows that, for sufficiently large $n$,
$$\sum_{r=r_n+1}^\infty E\left[ \left( \hat U_r^+ \right)^p \right] \leq C_{\beta,p} \sum_{r=r_n+1}^\infty \left( E[\tilde N ] E\left[ e^{\beta (\chi - \tilde \tau)} \right]  \right)^r \leq C_{\beta,p} \sum_{r=r_n+1}^\infty (\rho_\beta+\epsilon)^r = O\left( (\rho_\beta+\epsilon)^{r_n} \right)$$
as $n \to \infty$, since $\rho_\beta + \epsilon < 1$. 

The proof involving $W$ and the $\{U_r\}_{r \geq 0}$ is essentially the same and is therefore omitted. 
\end{proof}

\bigskip
 
The following result regarding the contribution of all the paths with multiple labels in the predecessor graph is the most technical one in the paper, since it is where the subtle dependence introduced by the merging of paths plays a role. 

\begin{lemma} \label{L.RedundantLabels}
For $r \geq 1$ define $M_r = \{ {\bf i} \in \hat A_r: {\bf i} \sim {\bf j} \text{ for some }{\bf j} \neq {\bf i} \}$ to be the set of labels  in the predecessor graph at graph distance $r$ from the tagged job belonging to jobs with multiple labels. Then, for any $\beta > 0$ and $\tilde \rho_\beta = E[\tilde N] E\left[ e^{\beta(\chi  -\tilde \tau)} \right]$, 
$$E\left[ \sum_{{\bf i} \in M_r} e^{\beta \hat S_{\bf i}} \right]  \leq  \frac{r}{n} E\left[ \tilde N^2 \right]  (\tilde \rho_\beta)^{r}  .$$
\end{lemma}

\begin{proof}
We first write for $r \geq 1$, 
\begin{align*}
E\left[ \sum_{{\bf i} \in M_r} e^{\beta \hat S_{\bf i}} \right] &= \sum_{{\bf i} \in \mathbb{N}_+^r} E\left[ e^{\beta \hat S_{\bf i}} 1({\bf i} \in M_r) \right].
\end{align*}
Now note that along a path all jobs are different and therefore the service requirements $\{ \chi_{{\bf i}|k} \}_{k =1}^r$ are i.i.d. with distribution $B$, and are independent of the $\{ \hat N_{{\bf i}|k} \}_{k=0}^{r-1}$. The inter arrival times $\{ \hat \tau_{{\bf i}|k} \}_{k=1}^r$ do depend on the $\{ \hat N_{{\bf i}|k} \}_{k=0}^{r-1}$ in the sense that a large number of jobs in the predecessor graph at the time a job arrives increases its probability of being an immediate predecessor to more than one job, and therefore influences the rate of the corresponding $\hat \tau$. It follows that if we replace them by i.i.d. copies of $\tilde \tau$ independent of everything else, we obtain 
$$E\left[ e^{\beta \hat S_{\bf i}} 1({\bf i} \in M_r) \right] \leq \left( E\left[ e^{\beta (\chi - \tilde \tau)} \right] \right)^r P({\bf i} \in M_r).$$
To compute the last probability let $\mathcal{C}_{\bf i}$ denote the event that ${\bf i}$ is a common immediate predecessor to two or more jobs/fragments in the predecessor graph, and note that 
$$1({\bf i} \in M_r) = \sum_{s=1}^r 1( ({\bf i}|s-1) \in M_{s-1}^c, \, ({\bf i}|s) \in M_s) \leq \sum_{s=1}^r 1({\bf i} \in A_r, \, \mathcal{C}_{{\bf i}|s}) ,$$
and therefore,
\begin{align*}
\sum_{{\bf i} \in \mathbb{N}_+^r} P({\bf i} \in M_r) &\leq \sum_{s=1}^r \sum_{{\bf i} \in \mathbb{N}_+^{r}} E\left[ \prod_{k=0}^{r-1} 1(i_{k+1} \leq \hat N_{{\bf i}|k}) 1(\mathcal{C}_{{\bf i}|s} ) \right] .
\end{align*}

Next, let $\sigma_{\bf i } = T_1 - \hat \tau_{{\bf i}|1} - \dots - \hat \tau_{\bf i } $ denote the time at which job ${\bf i}$ arrived to the predecessor graph. Define $\mathcal{F}_t = \sigma( (\hat N_{\bf j}, s_{\bf j}, \chi_{\bf j} , \hat \tau_{\bf j} ): \sigma_{\bf j} > t )$ to be the sigma algebra containing the ``history" of the predecessor graph over the interval $(t, T_1]$, and note that $\mathcal{F}_{\sigma_{\bf i}}$ does not reveal whether $\mathcal{C}_{\bf i}$ occurred nor the value of $\hat N_{\bf i}$. Now note that for $1 \leq s \leq r-1$,
\begin{align*}
&\sum_{{\bf i} \in \mathbb{N}_+^{r}} E\left[ \prod_{k=0}^{r-1} 1(i_{k+1} \leq \hat N_{{\bf i}|k}) 1(\mathcal{C}_{{\bf i}|s} ) \right]  \\
&= \sum_{{\bf i} \in \mathbb{N}_+^{r}} E\left[ \prod_{k=0}^{s-1} 1(i_{k+1} \leq \hat N_{{\bf i}|k}) E\left[ \left. \prod_{k=s}^{r-1} 1(i_{k+1} \leq \hat N_{{\bf i}|k}) 1(\mathcal{C}_{{\bf i}|s} ) \right| \mathcal{F}_{\sigma_{{\bf i}|s}} \right] \right].
\end{align*}
Moreover, since $\{ \hat N_{{\bf i}|k} \}_{k = s}^{r-1}$ are independent of $\mathcal{F}_{\sigma_{{\bf i}|s}}$, then
\begin{align*} 
E\left[ \left. \prod_{k=s}^{r-1} 1(i_{k+1} \leq \hat N_{{\bf i}|k}) 1(\mathcal{C}_{{\bf i}|s} ) \right| \mathcal{F}_{\sigma_{{\bf i}|s}} \right] = E\left[ \left. 1(i_{s+1} \leq \hat N_{{\bf i}|s}) 1(\mathcal{C}_{{\bf i}|s} ) \right| \mathcal{F}_{\sigma_{{\bf i}|s}} \right] \prod_{k=s+1}^{r-1} P(i_{k+1} \leq \hat N_{{\bf i}|k}),
\end{align*}
with the convention that $\prod_{i=a}^b x_i \equiv 1$ if $a > b$. 

To analyze the last conditional expectation let $K_t$ be the number of pieces of jobs that are available at time $t$, where by available we mean that they they do not have an immediate predecessor in $(t, T_1]$. Note that the event $\mathcal{C}_{\bf j}$ is a function of $K_{\sigma_{\bf j}}$ and $\hat N_{\bf j}$ only; more precisely,  for any ${\bf j} \in \mathbb{N}_+^s$ we have
\begin{align*}
P(\mathcal{C}_{\bf j}  | \hat N_{\bf j}, K_{\sigma_{\bf j}} )  &=  \frac{ \binom{ K_{\sigma_{\bf j}} -1}{1} \binom{n-2}{\hat N_{\bf j} -2}}{\binom{ K_{\sigma_{\bf j}} }{1} \binom{n-1}{\hat N_{\bf j} -1}} \leq  \frac{\hat N_{\bf j}}{n}.
\end{align*}
It follows that
\begin{align*}
E\left[ \left.  1(j_{s+1} \leq \hat N_{\bf j}) 1( \mathcal{C}_{\bf j})  \right| \mathcal{F}_{\sigma_{\bf j} } \right] &= E\left[ \left.  1(j_{s+1} \leq \hat N_{\bf j}) P( \mathcal{C}_{\bf j} |\hat N_{\bf j}, K_{\sigma_{\bf j}})  \right| \mathcal{F}_{\sigma_{\bf j} } \right] \\
&\leq  E\left[ \left.1(j_{s+1} \leq \hat N_{\bf j}) \frac{\hat N_{\bf j}}{n} \right|  \mathcal{F}_{\sigma_{\bf j} } \right] = \frac{1}{n} E[\hat N_{\bf j} 1(j_{s+1} \leq \hat N_{\bf j})] . 
\end{align*}

It follows that, for $1 \leq s \leq r-1$, 
\begin{align*}
&\sum_{{\bf i} \in \mathbb{N}_+^{r}} E\left[ \prod_{k=0}^{r-1} 1(i_{k+1} \leq \hat N_{{\bf i}|k}) 1(\mathcal{C}_{{\bf i}|s} ) \right]  \\
&\leq \sum_{{\bf i} \in \mathbb{N}_+^r} E\left[ \prod_{k=0}^{s-1} 1(i_{k+1} \leq \hat N_{{\bf i}|k}) \right] \frac{1}{n} E[\hat N_{{\bf i}|s} 1(i_{s+1} \leq \hat N_{{\bf i}|s})] \prod_{k=s+1}^{r-1} P(i_{k+1} \leq \hat N_{{\bf i}|k}) \\
&= \frac{1}{n} \left( E[ \tilde N ] \right)^{r-1} E[ \tilde N^2 ].
\end{align*}

For $s = r$ note that the same arguments used above give
\begin{align*}
\sum_{{\bf i} \in \mathbb{N}_+^{r}} E\left[ \prod_{k=0}^{r-1} 1(i_{k+1} \leq \hat N_{{\bf i}|k}) 1(\mathcal{C}_{\bf i} ) \right] &= \sum_{{\bf i} \in \mathbb{N}_+^{r}} \prod_{k=0}^{r-1} P(i_{k+1} \leq \hat N_{{\bf i}|k}) P(\mathcal{C}_{\bf i}) \\
&\leq \sum_{{\bf i} \in \mathbb{N}_+^{r}} \prod_{k=0}^{r-1} P(i_{k+1} \leq \hat N_{{\bf i}|k}) \frac{E[ \hat N_{\bf i}]}{n}
= \frac{1}{n} \left( E[\tilde N] \right)^{r+1}.
\end{align*}
We conclude that
\begin{align*}
E\left[ \sum_{{\bf i} \in M_r} e^{\beta \hat S_{\bf i}} \right] &\leq \left( E\left[ e^{\beta(\chi-\tilde \tau)} \right] \right)^r \left( \sum_{s=1}^{r-1} \frac{1}{n} \left( E[ \tilde N ] \right)^{r-1} E[ \tilde N^2 ] +  \frac{1}{n} \left( E[\tilde N] \right)^{r+1}  \right) \\
&= \frac{1}{n} (\tilde \rho_\beta)^r \left( (r-1) \frac{E[ \tilde N^2]}{E[\tilde N]} + E[\tilde N] \right). 
\end{align*}
Noting that $1 \leq (E[\tilde N])^2 \leq E[\tilde N^2]$ completes the proof. 
\end{proof}

\bigskip

\begin{lemma} \label{L.MaximumsLp}
Let $X_1, X_2, Y_1, Y_2$ be nonnegative random variables and let $p \geq 1$. Then
$$\left( E\left[ \left| X_1 \vee X_2 - Y_1 \vee Y_2 \right|^p \right] \right)^{1/p} \leq \left( E\left[  X_2^p \right] \right)^{1/p} + \left( E\left[ Y_2^p \right] \right)^{1/p} + \left( E\left[ \left| X_1 - Y_1  \right|^p \right] \right)^{1/p}. $$
\end{lemma}

\begin{proof}
First note that for any real numbers $x_1, x_2$ we have that
$$x_1 \vee x_2  = (x_2 - x_1)^+ + x_1,$$
from where we obtain that for $y_1, y_2$ also real,
$$|x_1 \vee x_2 - y_1 \vee y_2| \leq  (x_2 - x_1)^+ + |x_1 - y_1| + (y_2 - y_1)^+.$$
Moreover, provided $x_1, y_1, x_2, y_2 \geq 0$, we have that
$$|x_1 \vee x_2 - y_1 \vee y_2| \leq x_2 + |x_1 - y_1| + y_2.$$ 
Substituting in the random variables and using Minkowski's inequality gives the result.  
\end{proof}

\bigskip

The following proposition contains the main coupling between the predecessor graph and its weighted branching tree approximation. Its proof relies on the bound provided by Lemma~\ref{L.RedundantLabels}.

\begin{prop} \label{P.GraphTreeCoupling}
Suppose that Assumption \ref{A.Nconditions} is satisfied.  Then, for $\hat \nu_k$, the probability measure of $\bigvee_{r=0}^k \hat U_r$, and $\tilde \nu_k$, the probability measure of $\bigvee_{r=0}^k \tilde U_r$, we have that for any $r_n \to \infty$, and any $p \geq 1$,
$$W_p( \hat \nu_{r_n}, \tilde \nu_{r_n}) \to 0 \qquad n \to \infty.$$
\end{prop}

\begin{proof}
From the definition of the Wasserstein metric, we need to construct a coupling of $\hat \nu_{r_n}$ and $\tilde \nu_{r_n}$ for which we can show that their $L_p$ distance converges to zero. We will do this by defining  a weighted branching tree  that will be very close to the predecessor graph restricted to predecessors at graph distance at most $r_n$ of the tagged job (i.e., whose labels are of the form ${\bf i} = (i_1, \dots, i_r)$ with $r \leq r_n$). 
To start, define $M_r = \{ {\bf i} \in \hat A_r: {\bf i}\sim {\bf j} \text{ for some } {\bf j} \prec {\bf i} \}$ as in Lemma \ref{L.RedundantLabels}.  To construct the weighed branching tree we proceed inductively starting from the tagged job. 

Let $\{ \tilde N_{\bf i}'  \}_{{\bf i} \in U}$ be a sequence of i.i.d. copies of $\tilde N$, let $\{ \chi_{\bf i}' \}_{{\bf i} \in U}$ be a sequence of i.i.d. copies of $\chi$, and let $\{ \tilde \tau_{\bf i}' \}_{{\bf i} \in U}$ be a sequence of i.i.d. exponential random variables with rate $\lambda_n^*$, all sequences independent of each other and of all other random variables used up to now. Next, let $\tilde A_0 = \{ \emptyset\} = \hat A_0$ and $\tilde N_\emptyset \equiv \hat N_\emptyset$, which defines $\tilde A_1 = \{ i: 1 \leq i \leq \tilde N_\emptyset\}$. In general, for $r \geq 1$ and each ${\bf i} \in \tilde A_r$, set 
\begin{align*}
\tilde N_{\bf i} &= \begin{cases} 
\hat N_{\bf i}, &  \text{if } {\bf i} \in M_{r}^c , \\
\tilde N_{\bf i}' , & \text{otherwise}, \end{cases} 
\qquad \text{and} \qquad 
\tilde X_{\bf i} = \begin{cases}
\hat X_{\bf i}, &  \text{if } {\bf i} \in M_r^c, , \\
\chi_{\bf i}' - \tilde \tau_{\bf i}', &  \text{otherwise}. \end{cases}
\end{align*}
Then, use the newly defined $\{\tilde N_{\bf i} \}_{{\bf i} \in \tilde A_r}$ to construct $\tilde A_{r+1} = \{ ({\bf i}, i_{r+1}): {\bf i} \in \tilde A_r, \, 1 \leq i_{r+1} \leq \tilde N_{\bf i} \}$. We point out that, by construction, $M_r^c \subseteq \hat A_r \cap \tilde A_r$.  Also, the $\{ (\tilde N_{\bf i}, \tilde X_{({\bf i},1)}, \dots, \tilde X_{({\bf i}, \tilde N_{\bf i})} ) \}_{{\bf i} \in U}$ are now i.i.d. with the same distribution as $(\tilde N, \tilde X_1, \dots, \tilde X_{\tilde N})$, and therefore define a weighted branching tree.

Recall from the beginning of the section that
$$\tilde S_\emptyset = 0, \qquad \tilde S_{\bf i} = \tilde X_{{\bf i}|1} + \tilde X_{{\bf i}|2} + \dots + \tilde X_{\bf i}, \quad {\bf i} \neq \emptyset,$$
and 
$$\tilde U_r = \bigvee_{{\bf i} \in \tilde A_r} \tilde S_{\bf i}.$$
We will show that $E\left[ \left| \bigvee_{r=0}^{r_n} \hat U_r - \bigvee_{r=0}^{r_n} \tilde U_r \right|^p \right] \to 0$ as $n \to \infty$. 

To this end, note that we can split the paths in $\hat U_r$ and $\tilde U_r$ as follows:
$$\hat U_r = \max\left\{ \bigvee_{{\bf i} \in M_r^c} \hat S_{\bf i}, \, \bigvee_{{\bf i} \in  M_r} \hat S_{\bf i}^+ \right\} \triangleq \max\left\{ \hat U_r^{(1)}, \, \hat U_r^{(2)} \right\},$$
and
$$\tilde U_r = \max\left\{ \bigvee_{{\bf i} \in \tilde A_r \cap M_r^c} \tilde S_{\bf i}, \, \bigvee_{{\bf i} \in \tilde A_r \cap M_r} \tilde S_{\bf i}^+ \right\} \triangleq \max\left\{ \tilde U_r^{(1)}, \, \tilde U_r^{(2)} \right\}.$$
Note that $\emptyset \in \hat A_r \cap M_r^c$, and therefore $\hat U_r^{(1)}$ and $\tilde U_r^{(1)}$ are nonnegative without having to add the positive parts to the corresponding $\hat S_{\bf i}$ and $\tilde S_{\bf i}$. Moreover, $\hat S_{\bf i} \equiv \tilde S_{\bf i}$ for ${\bf i} \in  M_r^c$ and therefore, $\hat U_r^{(1)} \equiv \tilde U_r^{(1)}$. It follows from Lemma~\ref{L.MaximumsLp} that
\begin{align*}
E\left[ \left| \bigvee_{r=0}^{r_n} \hat U_r - \bigvee_{r=0}^{r_n} \tilde U_r \right|^p \right] &= E\left[ \left| \max\left\{ \bigvee_{r=0}^{r_n} \hat U_r^{(1)}, \,  \bigvee_{r=0}^{r_n} \hat U_r^{(2)} \right\} - \max\left\{ \bigvee_{r=0}^{r_n} \tilde U_r^{(1)}, \,  \bigvee_{r=0}^{r_n} \tilde U_r^{(2)} \right\}  \right|^p \right] \\
&\leq \left( \left( E\left[ \left( \bigvee_{r=0}^{r_n} \hat U_r^{(2)}  \right)^p \right]  \right)^{1/p} + \left( E\left[ \left( \bigvee_{r=0}^{r_n} \tilde U_r^{(2)}  \right)^p \right]  \right)^{1/p}   \right)^p.
\end{align*}

To analyze the last two expectations we follow the same approach used in the proof of Lemma~\ref{L.LpConvergence} to obtain
\begin{align*}
E\left[ \left( \bigvee_{r=0}^{r_n} \hat U_r^{(2)}  \right)^p \right] \leq C_{\beta,p} \, \sum_{r=1}^{r_n} E\left[ \sum_{{\bf i} \in M_r} e^{\beta \hat S_{\bf i}} \right]  ,
\end{align*}
where $C_{\beta,p} = p E[\xi^{p-1}]/\beta^p$ with $\xi$ exponentially distributed with rate one, and $\beta > 0$ is the one from Assumption \ref{A.Nconditions} (ii). By Lemma \ref{L.RedundantLabels} we have that
$$E\left[ \sum_{{\bf i} \in M_r} e^{\beta \hat S_{\bf i}} \right] \leq \frac{r}{n}  E[\tilde N^2] \left( E[\tilde N] E\left[ e^{\beta (\chi - \tilde \tau)} \right] \right)^r.$$

By \eqref{eq:LimitingRho} we have that $E[\tilde N] E\left[ e^{\beta (\chi - \tilde \tau)} \right] \to \rho_\beta$ as $n \to \infty$. We conclude that, for any $0 < \epsilon < 1 - \rho_\beta$ and sufficiently large $n$, 
\begin{align*}
E\left[ \left( \bigvee_{r=0}^{r_n} \hat U_r^{(2)}  \right)^p \right] &\leq C_{\beta,p}  \sum_{r=1}^{r_n} \frac{r}{n}  E\left[ \tilde N^2  \right]  \left( E[\tilde N] E\left[ e^{\beta (\chi - \tilde \tau)} \right] \right)^{r} \\
&\leq \frac{C_{\beta,p} }{n} E\left[ \tilde N^2  \right]  \sum_{r=1}^{r_n} r ( \rho_\beta + \epsilon)^{r} = O\left( \frac{1}{n} E\left[ \tilde N^2 \right] \right).
\end{align*}

The proof that $E\left[ \left( \bigvee_{r=0}^{r_n} \tilde U_r^{(2)}  \right)^p \right] = O\left( E\left[ \tilde N^2 \right] n^{-1} \right)$ follows the same steps and is therefore omitted. We have thus shown that
$$E\left[ \left| \bigvee_{r=0}^{r_n} \hat U_r - \bigvee_{r=0}^{r_n} \tilde U_r \right|^p \right]  = O\left( \frac{1}{n} E\left[ \tilde N^2  \right] \right)$$
as $n \to \infty$, which in turn implies that $W_p( \hat \nu_{r_n}, \tilde \nu_{r_n}) \to 0$ by Assumption \ref{A.Nconditions} (iii). 
\end{proof}

%\begin{lemma} \label{L.MaxDifference}
%Let $\{x_i\}_{i \geq 1}$ and $\{y_i\}_{i \geq 1}$ be two sequences of real numbers. Then, for any $k \geq 1$, 
%$$\left| \max_{1\leq i \leq k} ( x_i + y_i ) - \max_{1 \leq i \leq k} x_i \right| \leq \max_{1 \leq i \leq k} |y_i|.$$
%\end{lemma}
%
%\begin{proof}{Proof.}
%Let $i^*$ be such that $x_{i^*} + y_{i^*} \geq x_i + y_i$ for all $1 \leq i \leq k$, and let $j^*$ be such that $x_{j^*} \geq x_i$ for all $1 \leq i \leq k$. Next suppose that $y_{i^*} \geq 0$ and note that
%\begin{align*}
%\left| \max_{1\leq i \leq k} ( x_i + y_i ) - \max_{1 \leq i \leq k} x_i \right| &=  x_{i^*} + y_{i^*}  - \max_{1 \leq i \leq k} x_i \leq y_{i^*} \leq \max_{1 \leq i \leq k} |y_i|.
%\end{align*}
%If $y_{i^*} \leq 0$ then write
%\begin{align*}
%\left| \max_{1\leq i \leq k} ( x_i + y_i ) - \max_{1 \leq i \leq k} x_i \right| &=  x_{j^*}  - \max_{1 \leq i \leq k} (x_i + y_i) \leq -y_{j^*} \leq \max_{1 \leq i \leq k} |y_i|.
%\end{align*}
%\end{proof}

The following, and last, preliminary result provides a coupling for two weighted branching trees. As pointed out earlier, this step is unnecessary if $f_n \equiv f$ for all $n$ sufficiently large. 

\begin{prop} \label{P.CouplingTwoTrees}
Suppose that Assumption \ref{A.Nconditions} (i)-(ii) is satisfied. Then, for $\tilde \nu_k$, the probability measure of $\bigvee_{r=0}^k \tilde U_r$, $\nu_k$, the probability measure of $\bigvee_{r=0}^k U_r$,  any $p \geq 1$ and any $r_n \to \infty$ such that 
$$\lim_{n \to \infty} \left|E[N] - E[\tilde N] \right| (E[N])^{r_n/p} \, r_n = 0,$$
we have that
$$W_p( \tilde \nu_{r_n}, \nu_{r_n}) \to 0 \qquad n \to \infty.$$
\end{prop}

\begin{proof}
We start by constructing a coupling of $\tilde \nu_{r_n}$ and $\nu_{r_n}$. Let $\{ \xi_{\bf i}\}_{{\bf i} \in U}$ and $\{ \zeta_{\bf i} \}_{{\bf i} \in U}$ be two independent sequences of i.i.d.~Uniform(0,1) random variables. Let $F_n(k) = \sum_{j=1}^k f_n(j)$, $F(k) = \sum_{j=1}^k f(j)$ and set
\begin{align*}
\tilde \tau_{\bf i} = -\frac{1}{\lambda_n^*} \log \xi_{\bf i} \qquad \text{and} \qquad \tau_{\bf i} = - \frac{1}{\lambda^*} \log \xi_{\bf i}, \\
\tilde N_{\bf i} = F_n^{-1}(\zeta_{\bf i}) \qquad \text{and} \qquad N_{\bf i} = F^{-1}(\zeta_{\bf i}),
\end{align*}
where $g^{-1}(t) = \inf\{ x \in \mathbb{R}: g(x) \geq t\}$ (this is the standard inverse transform construction). 

Now use the $\{\tilde N_{\bf i} \}$ and the $\{ N_{\bf i} \}$ to construct two branching trees according to $\tilde A_0 = \{ \emptyset \} = A_0$ and $\tilde A_r = \{ ({\bf i}, i_r): {\bf i} \in \tilde A_{r-1}, 1 \leq i_r \leq \tilde N_{\bf i} \}$, $A_r = \{ ({\bf i}, i_r): {\bf i} \in A_{r-1}, 1 \leq i_r \leq N_{\bf i} \}$ for $r \geq 1$. It remains to construct the sequences of service requirements. For the $\sim$ weighted branching tree we sample i.i.d. random variables $\{ \chi_{\bf i} \}_{{\bf i} \in U}$ having distribution $B$, independent of all other random variables, and set $\tilde X_{\bf i} = \chi_{\bf i} - \tilde \tau_{\bf i}$ for each ${\bf i} \in U$. Then let $X_{\bf i} = \chi_{\bf i} - \tau_{\bf i}$. 

To analyze the difference between the corresponding processes $\tilde U_r$ and $U_r$, we start by defining the notion of a miscoupling. We say that there has been a miscoupling at node ${\bf i} \in A_r$ if $N_{\bf i} \neq \tilde N_{\bf i}$ and $N_{{\bf i}|k} = \tilde N_{{\bf i}|k}$ for all $1 \leq k < r$. Next, define $C_r = \{ {\bf i} \in A_r: N_{{\bf i}|k} = \tilde N_{{\bf i}|k} \text{ for all } 1 \leq k < r \}$, which corresponds to the set of individuals in both trees that have no miscouplings along their paths. 

Following the same steps used in the proof of Proposition \ref{P.GraphTreeCoupling}, and with some abuse of notation, split the paths in $\tilde U_r$ and $U_r$ as follows:
$$\tilde U_r = \max\left\{ \bigvee_{{\bf i} \in \tilde A_r \cap C_r} \tilde S_{\bf i}^+, \, \bigvee_{{\bf i} \in \tilde A_r \cap C_r^c} \tilde S_{\bf i}^+ \right\} \triangleq \max\left\{ \tilde U_r^{(1)}, \, \tilde U_r^{(2)} \right\},$$
and
$$U_r = \max\left\{ \bigvee_{{\bf i} \in A_r \cap C_r} S_{\bf i}^+, \, \bigvee_{{\bf i} \in  A_r \cap C_r^c} S_{\bf i}^+ \right\} \triangleq \max\left\{  U_r^{(1)}, \, U_r^{(2)} \right\}.$$
By using Lemma \ref{L.MaximumsLp} we obtain
\begin{align}
E\left[ \left| \bigvee_{r=0}^{r_n} \tilde U_r - \bigvee_{r=0}^{r_n} U_r \right|^p \right] &= E\left[ \left| \max\left\{ \bigvee_{r=0}^{r_n} \tilde U_r^{(1)}, \,  \bigvee_{r=0}^{r_n} \tilde U_r^{(2)} \right\} - \max\left\{ \bigvee_{r=0}^{r_n}  U_r^{(1)}, \,  \bigvee_{r=0}^{r_n}  U_r^{(2)} \right\}  \right|^p \right] \notag \\
&\leq \left\{ \left( E\left[ \left( \bigvee_{r=0}^{r_n} \tilde U_r^{(2)}  \right)^p \right]  \right)^{1/p} +  \left( E\left[ \left( \bigvee_{r=0}^{r_n} U_r^{(2)}  \right)^p \right]  \right)^{1/p}   \right. \label{eq:Complements} \\
&\hspace{5mm} +\left.  \left( E\left[ \left| \bigvee_{r=0}^{r_n} \tilde U_r^{(1)} - \bigvee_{r=0}^{r_n}  U_r^{(1)}  \right|^p \right]  \right)^{1/p}  \right\}^p. \label{eq:Intersection}
\end{align}

The analysis of the two expectations in \eqref{eq:Complements} is very similar to the approach used in the proof of Proposition \ref{P.GraphTreeCoupling}, so we will skip many of the intermediate steps. First, we obtain that
$$E\left[ \left( \bigvee_{r=0}^{r_n} \tilde U_r^{(2)}  \right)^p \right] \leq C_{\beta,p}  \sum_{r=1}^{r_n} E\left[ \sum_{{\bf i} \in \tilde A_r \cap C_r^c} e^{\beta \tilde S_{\bf i}} \right] ,$$
where $C_{\beta,p}$ is a finite constant and $\beta > 0$ is the same one from Assumption \ref{A.Nconditions} (ii). To compute the expectation on the right hand side let $\mathcal{H}_k = \sigma( (\tilde N_{\bf i}, \tilde X_{({\bf i},1)}, \dots, \tilde X_{({\bf i}, \tilde N_{\bf i})}): {\bf i} \in \tilde A_s, \, 0 \leq s < k)$ for $k \geq 1$ and note that for $r \geq 2$,
\begin{align*}
a_r &\triangleq E\left[ \sum_{{\bf i} \in \tilde A_r \cap C_r^c} e^{\beta \tilde S_{\bf i}} \right] \\
&= E\left[ \sum_{{\bf i} \in \tilde A_{r-1} \cap C_{r-1}^c} \sum_{j=1}^{\tilde N_{\bf i}} e^{\beta \tilde S_{({\bf i},j)} }  + \sum_{{\bf i} \in \tilde A_{r-1} \cap C_{r-1}} \sum_{j=1}^{\tilde N_{\bf i}} e^{\beta \tilde S_{({\bf i},j)} } 1(\tilde N_{\bf i} \neq N_{\bf i}) \right] \\
&= E\left[ \sum_{{\bf i} \in \tilde A_{r-1} \cap C_{r-1}^c} e^{\beta \tilde S_{\bf i}} E\left[ \left. \sum_{j=1}^{\tilde N_{\bf i}} e^{\beta \tilde X_{({\bf i},j)} } \right| \mathcal{H}_{r-1} \right] \right] \\
&\hspace{5mm} + E \left[ \sum_{{\bf i} \in \tilde A_{r-1} \cap C_{r-1}} e^{\beta \tilde S_{\bf i}} E\left[ \left. \sum_{j=1}^{\tilde N_{\bf i}} e^{\beta \tilde X_{({\bf i},j)} } 1(\tilde N_{\bf i} \neq N_{\bf i}) \right| \mathcal{H}_{r-1} \right] \right] \\
&= \tilde \rho_\beta \, a_{r-1} + E\left[ \tilde N 1(\tilde N \neq N) \right] E\left[ e^{\beta (\chi - \tilde \tau)} \right] E \left[ \sum_{{\bf i} \in \tilde A_{r-1} \cap C_{r-1}} e^{\beta \tilde S_{\bf i}} \right] ,
\end{align*}
where $\tilde \rho_\beta = E[\tilde N] E\left[ e^{\beta (\chi -  \tilde \tau)} \right]$ and $(\tilde N, N) = (F_n^{-1}(\zeta), F^{-1}(\zeta))$  in the first of the last three expectations, with $\zeta$ Uniform(0,1). Letting $\mathcal{E}_n =E\left[ \tilde N 1(\tilde N \neq N) \right] E\left[ e^{\beta (\chi - \tilde \tau)} \right]$ and noting that
$$E \left[ \sum_{{\bf i} \in \tilde A_{r-1} \cap C_{r-1}} e^{\beta \tilde S_{\bf i}} \right] \leq E \left[ \sum_{{\bf i} \in \tilde A_{r-1}} e^{\beta \tilde S_{\bf i}} \right] = (\tilde \rho_\beta)^{r-1},$$
gives
$$a_r \leq \tilde \rho_\beta \, a_{r-1} + \mathcal{E}_n (\tilde \rho_\beta)^{r-1}.$$
Iterating this recursion $r-1$ times gives,
$$a_r \leq (\tilde \rho_\beta)^{r-1} a_{1} + (r-1) \mathcal{E}_n (\tilde \rho_\beta)^{r-1} =  \mathcal{E}_n r  (\tilde \rho_\beta)^{r-1}.$$
Since by \eqref{eq:LimitingRho} we have that $\tilde \rho_\beta \to \rho_\beta = E[N] E\left[ e^{\beta (\chi - \tau)} \right]$ as $n \to \infty$, then for $0 < \epsilon < 1-\rho_\beta$ and $n$ sufficiently large,
$$E\left[ \left( \bigvee_{r=0}^{r_n} \tilde U_r^{(2)}  \right)^p \right] \leq C_{\beta,p} \, \mathcal{E}_n \sum_{r=1}^{r_n}  r (\tilde \rho_\beta)^{r-1} \leq C_{\beta,p} \, \mathcal{E}_n \sum_{r=1}^{\infty}  r (\rho_\beta + \epsilon)^{r-1} = O\left( \mathcal{E}_n \right)$$
as $n \to \infty$. The proof for the expectation involving the $\{ U_r^{(2)}\}$ is symmetric with respect to the $^\sim$ notation, so we obtain
\begin{equation} \label{eq:MiscoupledPaths}
E\left[ \left( \bigvee_{r=0}^{r_n} \tilde U_r^{(2)}  \right)^p \right] + E\left[ \left( \bigvee_{r=0}^{r_n}  U_r^{(2)}  \right)^p \right] = O\left( \mathcal{E}_n \right).
\end{equation}
It remains to analyze the expectation in \eqref{eq:Intersection}.

The key idea to do this is to note that
$$\tilde U_r^{(1)} =  \max_{{\bf j} \in C_r} (S_{\bf j} + E_{\bf j}),$$
where 
$$E_{\bf j} = \left( 1 - \frac{\lambda^*}{\lambda_n^*} \right) \left( \tau_{{\bf j}|1} + \tau_{{\bf j}|2} + \dots + \tau_{{\bf j}} \right).$$
By using the inequality 
$$\left| \max_{1 \leq i \leq k} (x_i+y_i) - \max_{1 \leq i \leq k} x_i \right| \leq \max_{1\leq i \leq k} |y_i|,$$
for any sequences of real numbers $\{x_i\}_{i \geq 1}$ and $\{y_i \}_{i \geq 1}$ and any $k \geq 1$, we obtain
$$E\left[ \left| \bigvee_{r=0}^{r_n} \tilde U_r^{(1)} - \bigvee_{r=0}^{r_n}  U_r^{(1)}  \right|^p \right] \leq  E\left[ \bigvee_{r=0}^{r_n} \bigvee_{{\bf i} \in C_r} |E_{\bf i}|^p  \right] \leq E\left[ \bigvee_{{\bf i} \in A_{r_n}} |E_{\bf i}|^p  \right],$$
where for the last identity we used the observation that if ${\bf i} \in C_r$ for some $r \leq r_n$, then there is at least one ${\bf j} \in A_{r_n}$ such that $({\bf j} | r) = {\bf i}$ (recall that $f_n(0) = f(0) = 0$), and since all the $\{\tau_{\bf i}\}$ are nonnegative, $|E_{\bf i}| \leq |E_{\bf j}|$. To estimate the last expectation note that since the $\{\tau_{\bf i}\}_{{\bf i} \in U}$ are independent of the $\{ N_{\bf i} \}_{{\bf i} \in U}$, we have
$$E\left[ \bigvee_{{\bf i} \in  A_{r_n}} |E_{\bf i}|^p  \right] \leq E\left[ \sum_{{\bf i} \in  A_{r_n}} |E_{\bf i}|^p  \right] = E[| A_{r_n}| ] \left| \frac{1}{\lambda^*} - \frac{1}{\lambda_n^*} \right|^p E[ Y_{r_n}^p ],$$
where $Y_{r_n}$ is an Erlang random variable with parameters $(r_n, 1)$. Since $E[| A_{r_n}| ] = (E[N])^{r_n}$,
$$ \left| \frac{1}{\lambda^*} - \frac{1}{\lambda_n^*} \right| = \left| \frac{E[N]  - E[\tilde N]}{\lambda E[N] E[\hat N]} \right| \leq \frac{|E[N] - E[\tilde N]|}{\lambda}, $$
and 
$$E[ Y_{r_n}^p ] = \int_0^\infty \frac{x^{r_n + p -1} e^{-x}}{(r_n-1)!} dx = \frac{\Gamma(r_n+p)}{\Gamma(r_n)},$$
where $\Gamma(t)$ is the gamma function, we have that
\begin{align}
E\left[ \left| \bigvee_{r=0}^{r_n} \tilde U_r^{(1)} - \bigvee_{r=0}^{r_n}  U_r^{(1)}  \right|^p \right] &\leq \frac{|E[N] - E[\tilde N]|^p}{\lambda^p} (E[N])^{r_n} \frac{\Gamma(r_n+p)}{\Gamma(r_n)} \notag \\
&= O\left( \left|E[N] - E[\tilde N] \right|^p (E[N])^{r_n} \, r_n^p \right) \label{eq:piece3}
\end{align}
as $n \to \infty$, where in the last step we used that $\lim_{k \to \infty} \Gamma(k) k^\alpha/\Gamma(k+\alpha) = 1$ for any $\alpha \in \mathbb{R}$. 

Combining \eqref{eq:MiscoupledPaths} and \eqref{eq:piece3} with \eqref{eq:Complements} and \eqref{eq:Intersection}, gives
$$\left( E\left[ \left| \bigvee_{r=0}^{r_n} \tilde U_r - \bigvee_{r=0}^{r_n} U_r \right|^p \right] \right)^{1/p} = O\left( \mathcal{E}_n^{1/p} + \left|E[N] - E[\tilde N] \right| (E[N])^{r_n/p} \, r_n    \right),$$
where
$$\left|E[N] - E[\tilde N] \right| (E[N])^{r_n/p} \, r_n \to 0$$
as $n \to \infty$ by assumption. To see that $\mathcal{E}_n \to 0$ as well, note that
$$\mathcal{E}_n =  \frac{\tilde \rho_\beta}{E[\tilde N]} \sum_{k=1}^{m_n} k P(N \neq k | \tilde N = k) f_n(k) \leq \tilde \rho_\beta \max_{1\leq k \leq m_n} P(N \neq k| \tilde N=k) .$$

Since $\tilde \rho_\beta \to \rho_\beta < 1$, it only remains to show that the maximum on the right hand side converges to zero. To see this is the case let $R_k = \{ u \in (0,1): F_n^{-1}(u) = k\}$ and note that for any $1 \leq k \leq m_n$,
\begin{align*}
P(N \neq k | \tilde N = k) &= \int_{R_k} \left\{ 1( F^{-1}(u) \leq k-1) + 1(F^{-1}(u) \geq k+1) \right\} du \\
&= \int_{R_k} \left\{ 1(F^{-1}(u) \leq F_n^{-1}(u) - 1) + 1(F^{-1}(u) \geq F_n^{-1}(u) + 1) \right\} du \\
&= \int_{R_k} 1( |F^{-1}(u) - F_n^{-1}(u)| \geq 1) \, du \\
&\leq \int_0^1 | F^{-1}(u) - F_n^{-1}(u)| \, du, 
\end{align*}
where the last integral is the Wasserstein distance of order one ($W_1$) between distributions $f_n$ and $f$ (see, e.g., \cite{Kan_Rub_58, Ruschendorf_85}), which converges to zero since $f_n \Rightarrow f$ and $E[\tilde N] \to E[N]$ (by Assumption~\ref{A.Nconditions} (i)), which is equivalent to convergence in $W_1$ (see Theorem 6.8 in \cite{Villani_2009}). This completes the proof. 
\end{proof}

\bigskip

Now that we have all the convergence results for each of the pairs of probability measures involved in \eqref{eq:TriangleIneq}, we can give the proof of the many servers limit of Theorem \ref{T.Main}.

\bigskip

\begin{proof}[Proof of Theorem \ref{T.Main} {\bf (Many servers limit)}]
Let $\varphi(n) = \left| E[N] - E[\tilde N] \right|$, which converges to zero as $n \to \infty$ since $f_n$ is uniformly integrable by Assumption \ref{A.Nconditions}. Now let
$$r_n = \begin{cases}
 \frac{p}{2 \log E[N]} |\log \varphi(n)|, & \text{ if } E[N] > 1, \\
\varphi(n)^{-1/2}, & \text{ if } E[N]  = 1,
\end{cases}$$
and note that
$$\left| E[N] - E[\tilde N] \right| (E[N])^{r_n/p} r_n = \begin{cases}
\frac{p}{2\log E[N]} \varphi(n)^{1/2}   |\log \varphi(n)|, & \text{ if } E[N] > 1, \\
\varphi(n)^{1/2}, & \text{ if } E[N] = 1,
\end{cases}$$
which converges to zero as $n \to \infty$ in both cases. 

That $\lim_{n \to \infty} W_p(\mu_n \mu) = 0$ as required is an immediate consequence of \eqref{eq:TriangleIneq} combined with Lemma~\ref{L.LpConvergence}, Proposition~\ref{P.GraphTreeCoupling}, and Proposition~\ref{P.CouplingTwoTrees}. 
\end{proof}

\bigskip

The last proof in the paper is that of Theorem \ref{T.CramerLundberg}. 

\begin{proof}[Proof of Theorem \ref{T.CramerLundberg}]
We need to verify the conditions of Theorem~3.4 in \cite{Jel_Olv_14}. The non-arithmetic condition is immediate from the observation that the $\{\tau_i\}$ are exponentially distributed and independent of $(N, \chi_1, \dots, \chi_N)$. The derivative and root conditions follow from the assumptions by noting that
$$E\left[ \sum_{i=1}^N e^{\theta(\chi_i - \tau_i)} (\chi_i - \tau_i) \right] = \frac{E[N] \lambda^*}{(\lambda^* + \theta)^2} \left( (\lambda^*+\theta) E\left[ e^{\theta \chi} \chi \right] - E\left[ e^{\theta \chi} \right] \right),$$
and
$$E\left[ \sum_{i=1}^N e^{\theta(\chi_i - \tau_i)} \right] = E[N] E\left[ e^{\theta \chi} \right] \lambda^*/(\lambda^*+\theta).$$
To verify condition 1 note that for $\theta > 1$, Lemma 4.1 in \cite{Jel_Olv_12a} gives (using $C_i \equiv 1$ for all $i$),
\begin{align*}
E\left[ \left( \sum_{i=1}^N e^{\chi_i - \tau_i} \right)^\theta \right] &\leq \left( E[ e^{(p-1)(\chi-\tau)} ] \right)^{\theta/(p-1)} E\left[ N^\theta \right] + E\left[ \sum_{i=1}^N e^{\theta(\chi_i - \tau_i)} \right] \\
&\leq E\left[ e^{\theta (\chi-\tau)} \right] E[N^\theta] + 1 \qquad \text{(by Jensen's inequality),}
\end{align*}
where $p = \lceil \theta \rceil$. For $0 < \theta \leq 1$, the same arguments give for any $0 < \epsilon < 1$,  
\begin{align*}
E\left[ \left( \sum_{i=1}^N e^{\theta (\chi_i - \tau_i)/(1+\epsilon)} \right)^{1+\epsilon} \right] &\leq \left( E[ e^{\theta(\chi-\tau)/(1+\epsilon)} ] \right)^{1+\epsilon} E\left[ N^{1+\epsilon} \right] + E\left[ \sum_{i=1}^N e^{\theta(\chi_i - \tau_i)} \right] \\
&\leq E\left[ e^{\theta (\chi-\tau)} \right] E[N^{1+\epsilon}] + 1.
\end{align*}
Since $E[N^{\theta \vee (1+\epsilon)}] < \infty$ by assumption, all the conditions of the theorem are satisfied. 
\end{proof}

\bibliographystyle{plain}
\bibliography{Clouds}

\end{document}